\documentclass[11pt, reqno]{article}
\usepackage{amsmath}
\usepackage{amsfonts}
\usepackage{amssymb}
\usepackage{amsthm}
\usepackage{float}
\usepackage{color}
\usepackage{graphicx}
\usepackage{psfrag}
\usepackage{bm}

\usepackage[left=2.5cm,right=2.5cm,top=2.5cm,bottom=3.5cm,includefoot]{geometry}

\newtheoremstyle%
 {greenthm}%
 {}{}%
 {\color{myblue}}
 {}%
 {\color{red}\bfseries}%
{\color{red}.}%
 { }{}

\newtheorem{theorem}{Theorem}
\newtheorem{lemma}[theorem]{Lemma}

\newtheorem{definition}[theorem]{Definition}

\newtheorem{result}[theorem]{Result}
\theoremstyle{remark}

\newtheorem{remark}[theorem]{Remark}
\theoremstyle{greenthm}

\newcommand{\supp}[1]{{\text{supp}}({#1})}
\newcommand{\dist}{\mathrm{dist}}
\renewcommand{\S}{\mathcal{S}}

\newcommand{\delt}{\varepsilon}

\newcommand{\dif}{\, \mathrm{d}}

\newcommand{\Rr}{\mathbb{R}}
\newcommand{\Nn}{\mathbb{N}}
\newcommand{\Zz}{\mathbb{Z}}

\newcommand{\Dgauss}{\mathcal{D}}
\newcommand{\Dcurv}{\mathcal{D}_{curv}}

\newcommand{\T}{\mathcal{T}}

\newcommand{\grid}{\mathbb{X}}

\renewcommand{\L}{\Lambda}

\newcommand{\curvenorm}[1]{[#1 ]_{\alpha}}

\newcommand{\cost}{\mathrm{c}}

\newcommand{\gam}{\bm{\gamma}}
\newcommand{\G}{\bm{G}}
\newcommand{\EE}{\bm{E}}

\newcommand{\Ps}{\bm{\Psi}}
\newcommand{\Derr}{\mathcal{D}_{\Ps}}

\newcommand{\fiber}{\mathcal{K}}
\newcommand{\absfiber}{\tilde{\mathcal{K}}}

\newcommand{\ful}[1]{T_h^{\sharp} #1}

\newcommand{\notful}[1]{T_{h}^{\flat} #1}

\renewcommand{\H}{\mathcal{H}}
\newcommand{\B}{\mathcal{B}}

\newcommand{\M}{\Gamma}
\newcommand{\LLL}{3}% integer bigger than \LL
\newcommand{\Manoa}{M\=anoa}
\newcommand{\Hawaii}{Hawai\kern.05em`\kern.05em\relax i }

\usepackage[ruled,vlined]{algorithm2e}

\begin{document}

\title{Anisotropic Gaussian approximation in $L_2(\Rr^2)$}
\author{ W.~Erb 
\thanks{Dipartimento Matematica ``Tullio Levi-Civita'', Universit{\`a} degli Studi di Padova, 
Via Trieste 63, 35121 Padova, Italy},
 T.~Hangelbroek
\thanks{Department of Mathematics, University of \Hawaii   -- \Manoa,
 Honolulu, HI 96822, USA.   
Research supported by   grant DMS-1716927 from the National
    Science Foundation.},
    A.~Ron\thanks{Department of Computer Sciences, University of Wisconsin, Madison, USA.
    Research supported by grant DMS-1419103 from the National
    Science Foundation.}} 
\maketitle
%%%%%%%%%%%%%%%%%%%%%%%%%%%%%%
%%%%%%%%%%%%%%%%%%%%%%%%%%%%%%
%
%Begin Abstract
%
%%%%%%%%%%%%%%%%%%%%%%%%%%%%%%
%%%%%%%%%%%%%%%%%%%%%%%%%%%%%%
\begin{abstract}
Let $\Dgauss$ be the {\it dictionary of Gaussian mixtures},
i.e., of the functions created by all possible linear
change of variables, and all possible translations, 
of a single Gaussian in multi-dimension. The dictionary $\Dgauss$ is vastly
used in engineering and scientific applications to a degree that practitioners
often use it as their default choice for representing their  scientific
object. The pervasive use of $\Dgauss$ in applications hinges on the
scientific perception that this dictionary is {\it universal} in 
the sense that it is large enough, and its members are local enough in 
space as well as in frequency, to provide efficient approximation 
to ``almost all objects of interest". 
However, and perhaps surprisingly, only a handful of concrete theoretical
results are actually known on the ability to use Gaussian mixtures in lieu
of mainstream representation systems.
Previous work, by Kyriazis and Petrushev, \cite{KP}, and by two of us,
\cite{HR}, settled the isotropic case, 
when smoothness is defined via a Sobolev/Besov/Triebel-Lizorkin topology:
it is shown in these papers that the non-linear $N$-term approximation 
from the dictionary $\Dgauss$ (and as a matter of fact from a small subset of it) is as effective 
as the counterpart provided by wavelets. 

The present paper shows that, in 2D, Gaussian mixtures 
are effective in resolving anisotropic structures, too. 
In this setup, the ``smoothness class" is comprised of 2D functions that are 
sparsely represented using {\it curvelets}. 
An algorithm for $N$-term approximation from (again, a small subset of) 
$\Dgauss$ is presented, and the error bounds are then shown to be on par
with the errors of $N$-term curvelet approximation. 
The latter are optimal, essentially by definition. 
Our approach is based on providing effective approximation 
from $\Dgauss$ to the members of the curvelet system, 
mimicking the approach in \cite{DR,HR} where the mother wavelets are 
approximated. When the error is measured in the $1$-norm, 
this adaptation of the prior approach, combined with standard tools,
yields the desired results.
However, handling the $2$-norm case  is much more subtle and
 requires substantial new machinery: in  this case, the error
analysis cannot be solely done on the space domain: some of it has to be 
carried out on frequency. Since, on frequency, all members of $\Dgauss$ are centered at the origin,  
a delicate analysis for controlling the error there is 
needed.
\end{abstract}
%%%%%%%%%%%%%%%%%%%%%%%%%%%%%%
%%%%%%%%%%%%%%%%%%%%%%%%%%%%%%
%
%End Abstract
%
%%%%%%%%%%%%%%%%%%%%%%%%%%%%%%
%%%%%%%%%%%%%%%%%%%%%%%%%%%%%%

%%%%%%%%%%%%%%%%%%%%%%%%%%%%%%
%%%%%%%%%%%%%%%%%%%%%%%%%%%%%%
\section{Introduction}
%%%%%%%%%%%%%%%%%%%%%%%%%%%%%%
%%%%%%%%%%%%%%%%%%%%%%%%%%%%%%
%
%Intro: First Subsection
%
%%%%%%%%%%%%%%%%%%%%%%%%%%%%%%
%%%%%%%%%%%%%%%%%%%%%%%%%%%%%%
\subsection{The universal Gaussian mixture dictionary $\Dgauss$}
Given the $d$-dimension Gaussian $\phi:x\mapsto e^{-|x|^2}$,
we are interested in $N$-term nonlinear approximation from the {\em Gaussian mixture dictionary}
$$ \Dgauss :=  \{ \phi \circ M \mid M \in  GL(d,\Rr)\ltimes \Rr^d \};$$
that is, a typical element of the dictionary is of the form
$$x\mapsto e^{-|A(x - k)|^2},$$
where $A$ is an invertible $d\times d$ matrix and $k\in \Rr^d$.
The dictionary $\Dgauss$ is vast, and its members are superbly local in space as well as in frequency.
It seems plausible, then, that this dictionary is {\it universal} in the sense that it provides 
efficient approximation to all ``objects of interest''.
In particular, one should  expect the dictionary $\Dgauss$ to provide effective 
$N$-term approximation to functions that are 
smooth in the sense that they are sparsely encoded by one of the mainstream representation systems.
That said, and quite surprisingly, while Gaussian mixtures
are broadly used in engineering and scientific applications 
(including, e.g., the thousands of papers that cite either of \cite{Bilmes98}, \cite{Rasmussen}, or \cite{Xu}),
only a handful of concrete theoretical results  are known about
the efficacy of the dictionary $\Dgauss$  for $N$-term approximation. 

While, as we have just alluded to above, certain theoretical aspects of Gaussian
approximation are lacking, approximation by Gaussians {\it is} 
a common tool within Applied Mathematics.  For example, it is
used in scattered data approximation \cite{Wend,Mazya, Niyogi, Kuhn},
in PDEs (e.g., in computational chemistry \cite{GH,Carrington} and as
Gaussian beams for time dependent problems), in other areas of
numerical analysis 
(as shape functions for particle methods   \cite{huerta, WKLiu}
and as a component of the  non-equispaced fast Fourier transform 
\cite{Greengard}) and in statistics (specifically in
machine learning and, more specifically, 
in support vector machines \cite{Chapelle,StSc,YiZh}).  

Motivated by the fact that the theoretical underpinning
of the use of Gaussians in {\it anisotropic setups}
is rudimentary,  we focus in this paper on the
utilization of $N$-term Gaussian mixtures in anisotropic setups.
While we are unaware of existing results in this direction, there are 
related successful efforts for using Gaussians in anisotropic setups. Those,
however, go beyond the mixture dictionary $\Dgauss$,
by allowing the mixtures to be  {\it modulated}\footnote{In other words, 
the functions in the dictionary are of the form
$\{(e^{i\langle k, \cdot\rangle} \phi)\circ M\mid k\in \Rr^d, M\in GL(d,\Rr)\ltimes \Rr^d\}$.}: 
modulated, anisotropically scaled, rotated and translated  Gaussians
have been considered in \cite{Andersson,dHGR}. Other publications
that address Gaussians in anisotropic
setups include \cite{GH}, where
anisotropic Gaussian approximation  
(without rotations) has been used in very high dimensions 
(with the goal of obtaining dimension-independent convergence rates);
\cite{Dinh}, where
surface reconstruction based on
a mix of isotropic and anisotropically scaled RBFs is  considered;
\cite{Aiolli}, where 
the optimal selection of a single anisotropic dilation parameter 
in  Gaussians  SVM learning is  studied, and others.
In addition, anisotropic Gaussians have been used for 
numerical simulations in geosciences \cite{cheng}.

Less has been done on the utilization of Gaussians
for characterizing smoothness in the ensuing $N$-term approximation schemes.
An early result is Meyer's treatment \cite[Chapter 6.6]{M} of
 the bump algebra $B_{1,1}^d(\Rr^d)$ (consisting of infinite series of Gaussians)
with wavelets.  Niyogi and Girosi in \cite{Niyogi} give 
rates for $N$-term approximation by isotropic Gaussian for functions in a nearby space 
(similar to $B_{1,1}^d(\Rr^d)$).
The general {\it isotropic} setup in $d$-dimensions,
when smoothness is defined via a Sobolev/Besov/Triebel-Lizorkin topology
was treated by Kyriazis and Petrushev, \cite{KP}, 
and by two of us, \cite{HR}.
It is shown in these papers that the non-linear $N$-term approximation from the dictionary $\Dgauss$ 
(and as a matter of fact from a small subset of it) is as effective as the counterpart provided by wavelets.

%%%%%%%%%%%%%%%%%%%%%%%%%%%%%%
%%%%%%%%%%%%%%%%%%%%%%%%%%%%%%
%
%Intro: Second subsection
%
%%%%%%%%%%%%%%%%%%%%%%%%%%%%%%
%%%%%%%%%%%%%%%%%%%%%%%%%%%%%%
\subsection{Treating anisotropic setups: our main results}

We study in this paper the utility of the ``universal Gaussian mixture
dictionary'' $\Dgauss$ in the representation of anisotropic 2D objects.
Our ``smoothness classes'' are defined as 
sets of functions
that are sparsely represented by {\em curvelets}.
However,
our results easily extend to smoothness classes defined via alternative anisotropic systems 
(e.g., shearlets, or other parabolic molecules):
such extension is readily possible using the machinery and the 
results that were developed and established by Grohs and Kutyniok \cite{GK}.

It should be stressed that we do not develop a new, Gaussian-based, representation system in this paper. 
Rather, we describe directly an algorithm for $N$-term approximation using 
members of the Gaussian dictionary $\Dgauss$.
In this sense, our results here are essentially different from those that have been established recently 
by
de Hoop, Gr{\"o}chenig and Romero  in \cite{dHGR}.
Indeed, while \cite{dHGR} resolves similar classes of anisotropic objects, its setup differs from 
ours in two substantial aspects. First, the Gaussians in \cite{dHGR} are modulated, and therefore 
the representation in \cite{dHGR} employs functions that are  outside our ``universal" $\Dgauss$.
Second, the focus in \cite{dHGR}  is not on $N$-term approximation
but on  the construction of a representation system (a frame),
with the $N$-term approximation schemes derived from the frame construction via standard arguments.

We describe now our setup and approach, and state our main result.
At this introductory level, we assume the reader to have a passing familiarity with the curvelet system 
developed by Cand{\`e}s and Donoho \cite{CD}. In the interest of being self-contained, we have included
the relevant technical details of the 
curvelet construction later on  in the body of our paper.

Given $\alpha>1/2$, our 
curvelet-based smoothness space is comprised of functions that 
are defined as follows.
Given $f\in L_2(\Rr^2)$, we expand $f = \sum_{n=1}^{\infty} \omega_n \gamma_n$ in a  decreasing rearrangement 
curvelet series, i.e., 
$(\gamma_n)_{n\in \Nn}$ are members of the curvelet system, 
and $|\omega_n|\ge |\omega_{n+1}|$
for all $n \in \Nn$. Then our smoothness class $\mathcal{C}_\alpha$ is defined
by
$$f\in\mathcal{C}_\alpha \iff  |\omega_n| = \mathcal{O}(n^{-\alpha}).$$
The assertion in the theorem below deals with the efficient 
resolution of $f\in \mathcal{C}_{\alpha}$ by
Gaussian mixtures.
Our techniques and results, however, extend to other, related, smoothness classes.
For example, we discuss (below, after the theorem)
the {\em cartoon class} of Donoho and Cand{\`e}s. 
This class, to recall, consists of functions of the form 
$F = f_1+f_2\chi_{\Omega}$, with $f_1$ and $f_2$ both functions in $C^2$
and $\Omega$ is a compact set with $C^2$ boundary.
This class enjoys a slightly different 
decay property of the curvelet coefficients; 
namely, $|\omega_n| = \mathcal{O}(n^{-3/2} |\log n|^{3/2})$,
(see \cite[Theorem 1.2]{CD}).\footnote{Note that
the cartoon class is contained in $\mathcal{C}_{\alpha}$ for $\alpha<3/2$.} We will come back to this 
cartoon class after stating our main result.

%%%
%
%%%
\begin{theorem}\label{main_intro}
For $\alpha>1/2$, $f \in \mathcal{C}_{\alpha}$ and any  $N\in \Nn$, there 
exists a Gaussian mixture with $N$ components, $\T_N f= \sum_{j=1}^N a_j G_j$ 
(with $G_j \in \Dgauss$ for each $j$),
such that
\[\|\T_N f - f\|_2 \leq C_f {N}^{\frac12 -\alpha}.\]
The constant $C_f$ depends on $f$ (and on $\alpha$), but is independent of $N$.
\end{theorem}
%%%
%
%%%
This result should be compared with  
$N$-term approximation by  the curvelets themselves. 
For $f=\sum_{n=1}^{\infty} \omega_n \gamma_n\in \mathcal{C}_{\alpha}$ as above,
the $N$-term approximant that is obtained by simple thresholding, viz.,
$\sum_{n=1}^{N} \omega_n \gamma_n$, has its error bounded by a constant
multiple of
$\left(\sum_{n=N+1}^{\infty} \omega_n^2\right)^{1/2}\le C N^{\frac12 -\alpha}$. 

An analogous result for the cartoon class
can then be easily attained. In that case, the curvelet $N$-term
approximation is of error 
$\mathcal{O}(N^{-1} (\log N)^{1.5})$, \cite[Theorem 1.3]{CD}. 
Our result for an $N$-term Gaussian mixture is 
$\mathcal{O}(N^{-1} (\log N)^{1.5})$ as well
(see our remark in section \ref{sec=cartoon}).

%%%%%%%%%%%%%%%%%%%%%%%%%%%%%%
%%%%%%%%%%%%%%%%%%%%%%%%%%%%%%
%
%Intro: Third subsection
%
%%%%%%%%%%%%%%%%%%%%%%%%%%%%%%
%%%%%%%%%%%%%%%%%%%%%%%%%%%%%%
\subsection{Our Algorithm}
\def\m{m}

Our approximation scheme consists of two main components: 
a high order  Gaussian approximant for individual curvelets, 
and a budgeting strategy which details the number of Gaussians
that are used in our approximation of  the individual curvelets
in $\Dgauss$.

{\bf Step I: Approximation of individual curvelets.} 
This is auxiliary, and is independent of $f$ and $\alpha$. 
In that step, we develop approximation schemes by Gaussian mixtures 
for the elements $\gamma$ of the curvelet system itself: given a 
curvelet $\gamma$ and a `budget' of $M$ Gaussians, 
we provide in the first step an 
approximant $G_{\gamma,M}$ for $\gamma$. The approximant is a Gaussian 
mixture with $M$ Gaussians, i.e., a linear combination of $M$ 
members from the dictionary $\Dgauss$. This step is facilitated by 
the following crucial fact: the curvelet system is obtained by 
applying certain unitary operations to a smaller number of prototypes. 
{\it The Gaussian mixture dictionary is invariant under these unitary
operations}. Thus, if $\gamma=U\gamma'$, where $\gamma$ is a curvelet, 
$\gamma'$ the corresponding curvelet prototype, and $U$ the unitary 
transformation, then we define
$$G_{\gamma,M}:=UG_{\gamma',M},$$
and reduce in this way  the  
problem to finding effective 
approximations for the prototypes themselves.

{\bf Step II: Budgeting.} 
We assemble the $N$-term approximation, $\T_N f$, for $f$, 
from suitable approximations, $G_{\gamma,M}$, to the underlying
curvelets. We first expand $f$, as before, in its decreasing rearrangement:
$$f=\sum_{n=1}^\infty \omega_n\gamma_n.$$
Then, once $N$ is given, our goal is to approximate $f$ by a mixture
$\mathcal{T}_Nf$ of $N$ Gaussians. 
To this end, we partition our budget $N$ of
Gaussians into
$$N=\sum_{n=1}^\infty N(n),$$
of non-negative sub-budgets. We note that the 
above decomposition of $N$ is {\it universal}: it does not depend on $f$ and
it does not depend on $\alpha$. 
Specifically, if $N(n)\not=0$, one of the possible budgeting algorithms in this
paper takes the form
$$N(n)= \left\lfloor \frac{\sqrt{2}}{2}\sqrt{\frac{N}{m_n}} \right\rfloor,$$
with $\m_n = 2^{\lceil\log_2(n)\rceil}$.
For example, if $N=256$, then the first few terms in the sequence $(N(n))_n$
are 
$$(11,8,5,5,4,4,4,4,\ldots).$$
For those values $N(n)$ which exceed a threshold $N_0$ 
(which is the minimum number of Gaussians we use for approximating any curvelet), 
we approximate
the curvelet $\gamma_n$ by a mixture $G_{\gamma_n,N(n)}$ 
using the method that is discussed in {\bf Step I}.

{\bf The approximation scheme:} Given a function $f\in \mathcal{C}_\alpha$,
and an integer $N$, we construct the
Gaussian mixture $\T_N f$ from Theorem \ref{main_intro} by employing
the two steps described above: 
$$\mathcal{T}_N f:=\sum_n \omega_n G_{\gamma_n,N(n)},$$
with the summation being over all the terms that {\it are} approximated. 
Since the sub-budgets $(N(n))_n$ are non-increasing, we always approximate
the head of the series
$$f=\sum_{n=1}^\infty\omega_n\gamma_n,$$
and exclude from the approximation the tail of that series. Moreover, 
since the sub-budgets $N(n)$ depend only on $N$, the number of curvelets
that we approximate depends also only on $N$. 
The scheme $\mathcal{T}_N$, however,
is non-linear: the non-linearity
is due here to the fact that the $n$th curvelet $\gamma_n$  in the rearranged series depends on the 
function $f$, hence the actual curvelets that are approximated are 
$f$-dependent.

\subsection{Organization}
The paper is structured as follows.
In section 2, we present some background material on  the curvelet system, 
a more  technical overview of the $N$-term approximation
algorithm, and a more precise statement of the main theorem. We also state
two lemmas, Lemma \ref{lemma-1} and Lemma \ref{lemma-2}, 
that are pillars in the proof of the theorem.

Section 3  contains the proof of the theorem. More precisely,
it shows how to invoke Lemmas \ref{lemma-1} and \ref{lemma-2} in order
to yield the result, but does not contain the proofs of the lemmas.

Lemma \ref{lemma-1}, which describes how the approximation of individual
curvelets ({\bf Step I} in the algorithm) is carried out, is proved in 
section 4.

Lemma \ref{lemma-2}, which shows, roughly, that the system of individual 
errors forms a Bessel system (which is an exceptionally subtle
fact) is proved in section 5. 
Our proof  of this fact uses the theory of generalized
shift invariant systems (GSI theory); the argument also incorporates
ideas from the recent work of 
de Hoop, Gr{\"o}chenig, and Romero 
\cite{dHGR}
(indeed, a key technical result, specific to working with 
the curvelet system, has been adapted from their paper -- see Lemma \ref{lem=star_lemma} below).

The Appendix contains 
a technical construction of the curvelet system; the latter
details are needed for the proofs in section 5.

\section{Algorithm and main results}

%%%%%%%%%%%%%%%%%%%%%%%%%%%%%%
%%%%%%%%%%%%%%%%%%%%%%%%%%%%%%
%
%Section 2: First Subsection
%
%%%%%%%%%%%%%%%%%%%%%%%%%%%%%%
%%%%%%%%%%%%%%%%%%%%%%%%%%%%%%
\subsection{Basics on curvelets} \label{sec=basicscurvelets}
In this section, we outline some relevant details on curvelets. 
More specific details are collected in  Appendix \ref{append}. 
Our source for this ``background curvelet material" is \cite{CD}.

 We denote the curvelet dictionary by $\Dcurv$.
 This is  a countable subset of the Schwartz space $\mathcal{S}(\Rr^2)$.
 Moreover, each curvelet is band-limited; $\Dcurv$ is a tight frame for $L_2(\Rr^2)$.

 A given curvelet  $\gamma\in \Dcurv$ is associated with an index vector $\mu:=\mu(\gamma)$.
The index $\mu(\gamma)$ is a triplet:  
$\mu = (\mu(1),\mu(2),\mu(3))=:(j(\gamma),\ell(\gamma),k(\gamma))$
 -- it determines the scale $j:=j(\gamma)$, the orientation $\ell:=\ell(\gamma)$,
and the placement $k:=k(\gamma)$ of the curvelet $\gamma$ in $\Rr^2$. 

The curvelets from a single, fixed, 
scale level $j$ are generated by suitable unitary 
operations applied to a single prototype
curvelet $\gam^{(j)}$ (which we call, once the scale $j$ is fixed,
the {\em curvelet generator}, for the curvelets in that scale).
Each generator $\gam^{(j)}$ is a band-limited Schwartz function 
with non-negative Fourier transform.
The two other entries in $\mu(\gamma)$, for
a curvelet $\gamma$ at scale $j$,  are the rotation parameter
$\ell$ that varies over the integers in
$[0, 2^{\lfloor j/2\rfloor}-1]$, and the translation parameter
$k$ that varies over a $j$-dependent lattice $\grid_j\subset\Rr^2$:
\begin{equation} \label{eq:grid}
\grid_{j}  = \left\{\left( \frac{k_1}{\L_j}, \frac{k_2}{\lambda_j}\right)\mid (k_1,k_2)\in 2\pi \Zz^2 \right\},
\end{equation}
and where 
$\Lambda_j = { \delt_1(j) 2^j}$, $\lambda_j = \delt_2(j) 2^{\lfloor{j}/2\rfloor}$
are $j$ dependent lattice constants.\footnote{The $j$-dependent
grid correction factors 
$\delt_1$ and $\delt_2$ are defined in
Remark \ref{grid_correction_remark}, 
where it is shown that they  are positive and uniformly bounded from above, and away from $0$.  
See (\ref{correction_bounds}).}
The curvelet $\gamma$ has the form
\begin{equation}\label{eq:unitary_gam}
\gamma: x\mapsto
{|D_\gamma|}^{1/2} \gam^{(j)}\bigl(D_\gamma 
(R_{\gamma}^*x - k(\gam))\bigr), \quad x\in \Rr^2.
\end{equation}
Here, $D_{\gamma}$ is a dilation operator and $R_{\gamma}$ is a rotation operator 
(as is $R_{\gamma}^* = R_{\gamma}^{-1}$, naturally):
specifically, $D_\gamma$ is the parabolic scaling matrix
$$
D_\gamma := \begin{pmatrix} 2^{j(\gamma)}&0\\0&2^{\lfloor j(\gamma)/2\rfloor}\end{pmatrix},
$$
and $R_{\gamma}$ is a counter-clockwise rotation by the angle 
$\pi \ell(\gamma)\,2^{-\lfloor j(\gamma)/2\rfloor}$. 
We occasionally use the notation 
$D_j $ (with $j:=j(\gamma)$) in place of $D_{\gamma}$, and use
$R_{j,\ell}$ or $R_\mu$ in place of $R_{\gamma}$ (with $\mu:=\mu(\gamma)=:
(j,\ell,...))$.\footnote{ The dilation $D$ depends only on the scale $j$, hence
it could be embedded into the definition of $\gam^{(j)}$. However, in our 
definition of the curvelet generators, they are all similar to each other 
and have some uniformity, which we need to exploit later.}

We make frequent use of the following two facts
\begin{itemize}
\item The map $\mu:\Dcurv \to \{(j,\ell,k)\mid j\in \Nn,\  \ell\in \{0,\dots, 2^{\lfloor j/2\rfloor}-1\},\  k\in \grid_j\}$ 
from curvelets to their indices is a bijection.
\item Every curvelet is obtained from its generator via a unitary 
transformation:
$\gamma = U_{\gamma} \gam^{(j)}$
where $U_{\gamma}:L_2(\Rr^2)\to L_2(\Rr^2)$ is the unitary operator defined as
$$U_{\gamma} :F \mapsto {|D_\gamma|}^{1/2}F\bigl(D_{\gamma}  
(R_{\gamma}^*\cdot - k(\gamma))\bigr).$$
\end{itemize}

%%%%%%%%%%%%%%%%%%%%%%%%%%%%%%
%%%%%%%%%%%%%%%%%%%%%%%%%%%%%%
%
%Section 2: Second Subsection
%
%%%%%%%%%%%%%%%%%%%%%%%%%%%%%%
%%%%%%%%%%%%%%%%%%%%%%%%%%%%%%
\subsection{The anisotropic smoothness class $\mathcal{C}_\alpha$}\label{curvelet_class}

Throughout this article, we make two assumptions on the target function $f$:
membership in $L_{2}$ (corresponding to the fact that the error is measured
in $L_2$), and
a sparsity condition on the curvelet coefficients.
Since the curvelets  
$\Dcurv$
form a tight frame, we can expand any  
$f \in L_2(\Rr^2)$ as $f = \sum_{\gamma\in \Dcurv} \omega(\gamma) \gamma$
with the expansion coefficients
$$\omega(\gamma) = \langle f,\gamma\rangle.$$
We denote by $\omega^* = (\omega_n)_{n\in \Nn}$ 
the decreasing rearrangement of the coefficients $(\omega(\gamma))_\gamma$,
i.e.,
$$f=\sum_{n=1}^\infty\omega_n\gamma_n,$$
with $\gamma_n$ the curvelet with the $n$th largest coefficient in the expansion
of $f$, and $\omega_n:=\omega(\gamma_n)$.
Given $\alpha>1/2$, we assume then that $f\in\mathcal{C}_\alpha$, which means,
by definition, that
\begin{equation} \label{eq:cartoonclass}
|\omega_n| \leq C n^{-\alpha},\quad \forall n.
\end{equation}
It is not hard to see that $ \mathcal{C}_{\alpha}$ is a vector space 
(indeed, it is the space of functions whose curvelet coefficients lie in the Lorentz space 
$\ell_{\frac{1}{\alpha}, \infty}$).
For a given  function $f\in \mathcal{C}_{\alpha}$, we define  ${\curvenorm{f}}$ 
to be the smallest constant $C$ for which (\ref{eq:cartoonclass}) is valid. 
It is straightforward then to obtain the $N$-term approximation
rate by curvelets
for $f\in \mathcal{C}_\alpha$: with $s_N(f):=\sum_{n=1}^N\omega_n\gamma_n$,
one readily obtains the following {\it sharp} estimate:
$\|f - s_N(f)\|_2 \le  N^{\frac12 -\alpha}{\curvenorm{f}}/\sqrt{2\alpha-1}$. 
%
%

%%%%%%%%%%%%%%%%%%%%%%%%%%%%%%
%%%%%%%%%%%%%%%%%%%%%%%%%%%%%%
%
%Section 2: Third Subsection
%
%%%%%%%%%%%%%%%%%%%%%%%%%%%%%%
%%%%%%%%%%%%%%%%%%%%%%%%%%%%%%
\subsection{The Gaussian mixture approximation algorithm}\label{sec=algorithm}

Given a `budget' $N$, and a function $f\in\mathcal{C}_\alpha$,
we want to
approximate $f$ by an expansion of the form
$$
\T_N f = \sum_{n=1}^{m^*} \omega_n G_{\gamma_n, N(n)},
$$
where $m^*$ is an integer -- the index of the terminal curvelet to be approximated,
$G_{\gamma_n, N(n)}$ is a linear combination of $N(n)$ 
Gaussians from the Gaussian dictionary
$\mathcal{D}$ - a combination that approximates the curvelet $\gamma_n$, and
each $N(n)$ represents the portion of the full budget
$N$ which we devote to approximating the
single curvelet $\gamma_n$.
Our budgeting algorithm, to be specified later, satisfies the following:
\begin{enumerate}
\item The sub-budgets sum to (at most) the total budget: $\sum_{n=1}^{m^*} N(n) \le N$.
\item  The sub-budgets are constant (i.e., are all the same)
on each dyadic interval $2^{\nu-1}< n\le 2^{\nu}$; so $N(n) = N(2^\nu)$, for
every such $n$.
\item There is a minimal sub-budget $N_0\ge 1$ 
which represents the smallest positive 
investment\footnote{To satisfy the requirements of our individual approximation scheme ({\bf Step I}), 
we may take $N_0  = 3$. 
In principle, one could replace the individual approximation scheme we use (described in section 4), 
with another requiring a different value of $N_0$.
}
we can make in a curvelet. 
Thus $N(n)\ge N_0$ for each $n\le m^*$.
\end{enumerate}
Another aspect of our analysis below is that the sub-budgets are monotone with $|\omega_n|$. 
As a consequence, the condition $N(m^*)\ge N_0$ is
sufficient to guarantee item 3 above.

\subsubsection{Step I: approximating individual curvelets}\label{subsubsec=step2}
For a given $M \in\Nn$, we approximate the curvelet $\gamma$ by 
an $M$-term linear combination of Gaussians, which we call $G_{\gamma,M}$.
To this end, we first approximate, with $j:=j(\gam)$,
the generator
$\gam^{(j)}$ by $M$ members of $\mathcal{D}$.
We denote by $\G_M^{(j)}$ the $M$-term approximant to the generator $\gam^{(j)}$.
 In fact, with
$\phi:x\mapsto e^{-|x|^2}$, we use in this initial approximation
only {\it translations} of $\phi$: at most $M$ translations.
Now, $\gamma$ is obtained from $\gam^{(j)}$ by a unitary operator $U_\gamma$ 
(a composition of translation, rotation and dilation, as described in section \ref{sec=basicscurvelets}):
$\gamma= U_\gamma \gam^{(j)}.$
So, we define
\begin{equation} \label{eq:onetermgaussian}
G_{\gamma,M}:=U_\gamma \G_M^{(j)}.
\end{equation}
Obviously, $G_{\gamma,M}$ is a linear combination of (at most) $M$ functions,
each obtained by applying a unitary change of variables to a shift of $G$, 
hence each is a member of $\mathcal{D}$. So, $G_{\gamma,M}$ is a Gaussian mixture
with $M$ components.
The exact details of this construction, as well as an analysis  of the error 
are given in section \ref{sec:individualerror}.

\subsubsection{Step II: Investment strategy}
Given (any, but fixed) $0<\beta<1$, 
and  a total budget $N>0$, we  define a cost function $\cost:\Nn\to [0,\infty)$ 
 \begin{equation}\label{cost}
 \cost(\m) :=  
C(\beta) \left(\frac{N}{\m}\right)^{\beta},
\end{equation}
where $C(\beta): = (1-\beta)^{\beta} $.
This  cost function roughly determines the number of Gaussians to invest  in the curvelet $\gamma_{n}$.
For $\nu \in \Nn$, set $\M_{\nu}:= \{n\in \Nn \mid  2^{\nu-1} < {n} \le 2^{\nu}\}.$
Define
 \begin{equation} \label{eq:investmentstrategy}
 N(n) := \lfloor \cost(2^\nu) \rfloor ,\qquad n\in \M_{\nu}.
\end{equation}
It follows that $ N(n)\le \cost(n)\le 2^{\beta} N(n)$ for all $n\in \Nn$.
  
With $N_0\ge 1$ the minimal number
 of Gaussians we may use for approximating
an underlying curvelet, we set
\begin{equation}\label{terminal_index}
%\m^*:= \frac{1-\beta}{ 2N_0^{1/\beta}}N;
\m^*:= \max 
\left\{2^{\nu}\mid  2^\nu \le \frac{1-\beta}{ N_0^{1/\beta}}N\right\}.
\end{equation}
This is the terminal index which distinguishes the principal part of the curvelet expansion 
(containing the terms that we approximate) and the tail. %(which will be part of the error).
 With this choice,
$$
  N(m^*)  = \lfloor c(m^*)\rfloor \ge 
  \lfloor N_0 \rfloor  = N_0.
$$
Thus $N(m)\ge N_0$ for all $m\le  m^*$. 
It then follows that
$$ 
\sum_{m=1}^{m^*} N(m) 
\le
 C(\beta)N^{\beta} \sum_{m=1}^{m^*} m^{-\beta}
\le \frac{ C(\beta)N^{\beta}}{1-\beta} (m^*)^{1-\beta}
= \frac{C(\beta)}{1-\beta} \frac{(1-\beta)^{1-\beta}}{(N_0^{1/\beta})^{1-\beta}} N \le N.
$$
In other words, we do not outspend our budget.

\smallskip
We summarize the steps of our scheme in the following Algorithm \ref{algorithm1}.

\begin{center}
\begin{algorithm}[H] \label{algorithm1}

\vspace{4mm}

\KwIn{A budget $N \in \Nn$, and  the decreasing rearrangement of
a function $f \in \mathcal{C}_{\alpha}$.
}

\vspace{2mm}

\textbf{Investment:} 
Distribute the budget $N$ over the single curvelets $\gamma_n$ in parts 
$N(n)$ based on the strategy
\eqref{eq:investmentstrategy}.  \\[2mm]

\textbf{Approximation of curvelet generators:} For each feasible investment
$N(n)$ (i.e., any sub-budget for which $N(n)\ge N_0$)
construct
$\G_{N(n)}^{(j)}$
 the $N(n)$-term Gaussian approximation
 to the curvelet generator $\gam^{(j)}$, with $j:=j(\gam_n)$ -
the scale of the $n$th curvelet in the decreasing rearrangement of
$f$.\footnote{Note that this Gaussian approximation coincides for any two 
curvelets that have identical budget and are in the same scale}\\[2mm]

\textbf{Individual curvelet approximation:} For each $\gamma_n$ with $n\le m^*$, 
construct the $N(n)$-term linear combination $G_{\gamma_n,N(n)}$ to $\gamma_{n}$ by
using formula (\ref{eq:onetermgaussian}) and the relevant approximant from the previous step. \\[2mm]

\textbf{Approximation procedure:} 
 Form the full Gaussian approximation of $f$  as
 \[ 
 \T_N f = \sum_{n=1}^{m^*} \omega_n G_{\gamma_n,N(n)}.
 \]
 \\%[2mm]
\caption{$N$-term curvelet approximation of $f$ using Gaussians}
\end{algorithm}
\end{center}
%
%

%%%%%%%%%%%%%%%%%%%%%%%%%%%%%%
%%%%%%%%%%%%%%%%%%%%%%%%%%%%%%
%
%Section 2: Fourth Subsection
%
%%%%%%%%%%%%%%%%%%%%%%%%%%%%%%
%%%%%%%%%%%%%%%%%%%%%%%%%%%%%%
\subsection{The main result}
The main result of this article reads as follows:
%%%
%
%%%
\begin{theorem} \label{thm-main}
For $\alpha>1/2$ and $f \in \mathcal{C}_{\alpha}$ with decreasing rearrangement curvelet
representation
$$f=\sum_{n=1}^\infty \omega_n\gamma_n,$$
and a given budget $N$, there exists an 
$N$-term Gaussian approximation $\T_N f$ of $f$ such that
\[
\|\T_N f - f\|_2 \leq C \curvenorm{f}  {N}^{\frac12 -\alpha}.
\]
The constant $C$ is independent of $N$ and of $f$.
The approximant $\T_Nf$ has the form
\[ 
\T_N f = \sum_{n=1}^{m^*} \omega_n G_{\gamma_n,N(n)},
\]
where $G_{\gamma_n,N(n)}$ is a linear combination of $N(n)$ Gaussians 
based on the construction \eqref{eq:onetermgaussian}, 
and $N(n)$ is chosen according to the strategy \eqref{eq:investmentstrategy}.
In particular, the  number of approximated curvelets $m^*$ is independent of $f$
and the sub-budgets $N(n)$ are independent of $f$ and $\alpha$.
\end{theorem}
%%%
%
%%%

The proof of Theorem \ref{thm-main} is based on the following two results 
that are derived separately in the upcoming sections 4 and 5.
The first result estimates the individual errors for the approximation
of the curvelet generators $\gam^{(j)}$ with an appropriate linear combination 
$\G_{j, M}$ of 
$M$ Gaussians.

%%%%%%%%%%%%%%%
%
%%%%%%%%%%%%%%%
\begin{lemma}[Individual curvelet approximation] \label{lemma-1}
For a given budget $M$, and a scale $j$,
there exists a linear combination $\G_{j, M}$ of $M$ Gaussians such that
 $\EE_{j,M} := \gam^{(j)} - \G_{j,M}$
obeys the following estimate:
for every $K$, 
there is a constant $C_K$ so that for all $M>0$,  $j\ge 4$ and $\xi=(\xi_1,\xi_2)\in \Rr^2$, the inequality
  \begin{equation}\label{one_term_rotation_error_mod}
  |\widehat{\EE_{j,M}}(\xi)|
 \le 
 C_K { M^{-K}} \min\bigl(|\xi_1| ,1\bigr)^{2}
 \left(1+  | \xi  | \right)^{-\LLL},
 \end{equation}
 holds, 
 while for  $j<4$ we have $  |\widehat{\EE_{j,M}}(\xi)|
 \le 
 C_K { M^{-K}} 
 \left(1+  | \xi  | \right)^{-\LLL}.
$
\end{lemma}
%%%%%%%%%%%%%%%%
%
%%%%%%%%%%%%%%%%

%%%%%%%%%%%%%%%%
%
%%%%%%%%%%%%%%%%
\begin{definition}
Letting $\gamma = U_{\gamma} \gam^{(j)}$, as discussed in  section \ref{sec=basicscurvelets}, the function
$G_{\gamma,M} := U_{\gamma} \G_{j,M}$ is an $M$-term Gaussian mixture.
We refer to $E_{\gamma,M} := \gamma - G_{\gamma,M}$ as
the {\em $M$-term individual error} (of the curvelet $\gamma$).
\end{definition}
%%%%%%%%%%%%%%%%
%
%%%%%%%%%%%%%%%%

Our second major  lemma  gives conditions on a family indexed by $\Dcurv$ to be a Bessel system.
We will use it to show that, for a fixed $M$, the individual error functions
in the $M$-term approximations of the curvelets, when properly normalized,
form a Bessel system.
%%%%%%%%%%%%
%
%%%%%%%%%%%%
\begin{lemma}[Bessel system] \label{lemma-2}
Suppose $\{\Ps_{j}\mid j\in \Nn\}$ is a family of functions  in $L_2(\Rr^2)$  satisfying the following:
there exists a constant  $C$,  so that for all $j\ge 4$, and all $\xi=(\xi_1,\xi_2)\in\Rr^2$,
$$|\widehat{\Ps_{j}}(\xi)|\le C \min(|\xi_1|^2,1) (1+|\xi|)^{-\LLL},$$
while $|\widehat{\Ps_{j}}(x)|\le C (1+|\xi|)^{-\LLL}$ holds for $j< 4$.
Then there is a constant $\varTheta$ (depending only on $C$) so that 
the family
$
\Derr:=
\{ 
  {\psi}_{\gamma} := U_{\gamma} \Ps_{j(\gamma)}\mid \gamma \in \Dcurv, 
  \gamma=U_{\gamma} \gam^{(j(\gamma))}\}
$
forms a Bessel system with Bessel constant $\varTheta$. 
In other words, the bound
$$
\left\| \sum_{\gamma \in \Dcurv}\omega_{\gamma} 
{\psi}_{\gamma}\right\|_{L_2(\Rr^2)} 
\leq 
\varTheta \|(\omega_{\gamma})_{\gamma}\|_{\ell_2}
$$
holds, for any square summable coefficients $(\omega_{\gamma})_{\gamma}$.
\end{lemma}
%%%%%%%%%%%%%%%%
%
%%%%%%%%%%%%%%%%

%%%%%%%%%%%%%%%%
%
%%%%%%%%%%%%%%%%
\begin{remark}We point out that the critical estimate on the individual error occurs on the Fourier domain. 
This is in contrast to the situation in \cite{HR}, which utilizes control on the (space-side) decay of the error 
(of course, in the present setting, we are only concerned with $N$-term error measured in $L_2(\Rr^2)$, which 
allows us to carry out our analysis on the Fourier domain).
\end{remark}
%%%%%%%%%%%%%%%%
%
%%%%%%%%%%%%%%%%

%%%%%%%%%%%%%%%%%%%%%%%%%%%%%%
%%%%%%%%%%%%%%%%%%%%%%%%%%%%%%
%
%Section 3
%
%%%%%%%%%%%%%%%%%%%%%%%%%%%%%%
%%%%%%%%%%%%%%%%%%%%%%%%%%%%%%
\section{Proof of the main theorem}\label{S:ProofMain}

The explicit construction $\G_M^{(j)}$ of the Gaussian approximation 
for the  curvelet generator $\gam^{(j)}$
is carried out in section \ref{sec:individualerror} in which Lemma \ref{lemma-1} is proven.
As mentioned in \eqref{eq:onetermgaussian}, an
$M$-term approximation of a single curvelet $\gamma = U_{\gamma} \gam^{(j)}$ 
is then obtained by dilating/rotating/translating the
Gaussian approximation $\G_{M}^{(j)}$ of the  curvelet generator: i.e., 
$G_{\gamma,M} = U_{\gamma} \G_M^{(j)}$. 
If the budgeting for the involved curvelets is based on the strategy \eqref{eq:investmentstrategy}, 
we obtain an approximation $\T_N f$ of $f$ consisting of at most $N$ Gaussians. 

We show in this section how to deduce Theorem \ref{thm-main} from 
Lemma \ref{lemma-1} and Lemma \ref{lemma-2}.

%%%%%%%%%%%%%%%%
%
%%%%%%%%%%%%%%%%
\begin{proof}[Proof of Theorem \ref{thm-main}]
Since, given the budget $N$, we approximate only the first $\m^\ast$ entries
in the curvelet expansion, the error is split accordingly:
$$
\T_Nf-f
=
\sum_{n=1}^{m^*}
  \omega_{n} ( \gamma_{n} - G_{\gamma_n,N(n)} ) 
  + 
\sum_{n=m^*+1}^{\infty} \omega_{n} \gamma_{n} .
$$
  By the triangle inequality we have
  $\|\T_Nf - f\|_2 \le I_+ + I_-$ where we introduce two parts of the error: the principal part
  $I_+ := 
  \|
  \sum_{n=1}^{m^*} \omega_{n} ( \gamma_{ n} - G_{\gamma_n,N(n)} ) 
  \|_2$,
  and the tail 
  $I_- :=\|\sum_{n=m^*+1}^{\infty} \omega_{n} \gamma_{n}  \|_2$.

 Because the curvelets form a Parseval frame 
 (see \cite[(2.10)]{CD} % 
 we have that
\begin{equation}\label{tail_estimate}
I_-     
\le
\sqrt{\sum_{n> m^*} |\omega_{n}|^2}
\le
{\curvenorm{f}}
\sqrt{\sum_{n> m^*}    n^{-2\alpha}}
\le 
 C{\curvenorm{f}}  (m^*)^{1/2 -\alpha}
 \le
 C{\curvenorm{f}}N^{1/2-\alpha}.
\end{equation}
The last inequality follows from  the fact that $\m^\ast$ coincides
with $N$, up to a multiplicative $f$-independent, $N$-independent constant:
see (\ref{terminal_index}).

To treat $I_+$, we  break the error into manageable pieces by partitioning  
$\{1,\dots, m^*\}$
into dyadic intervals:
i.e., we partition
$
\{1,\dots, m^*\} = \bigcup_{\nu = 0}^{\log_2 m^*} \M_{\nu}
$
 where 
$
\M_{\nu}= \{n\le m^* \mid  2^{\nu-1} < {n} \le 2^{\nu}\}.
$
Then we have at most $1+\log_2 N$ 
different sub-budgets, because for every $n\in\M_\nu$
we have
$N(n) = M_\nu := N(2^{\nu})$.
Applying the triangle inequality gives
$$
I_+ = 
\left\|
  \sum_{n=1}^{m^*} \omega_{n} E_{\gamma_n,N(n)}
\right\|_{L_2}
\le
\sum_{\nu = 0}^{\lfloor \log_2 N\rfloor}
 \left\|
   \sum_{n\in \M_{\nu}} \omega_{n} E_{\gamma_n,M_{\nu}}
 \right\|_{L_2}.
$$
Now fix $K>\min( (\alpha-1/2)/\beta,6)$, and set $\rho:= K \beta$. 
By Lemma \ref{lemma-1}, there is a constant $C$ so that for all $\xi\in \Rr^2$,
$$(M_{\nu})^K \widehat{\EE_{j,M_{\nu}}}(\xi)\le 
C
\begin{cases} 
  \min\bigl(|\xi_1| ,1\bigr)^{2}
 \left(1+  | \xi  |^2 \right)^{-\LLL}, &j\ge 4,\\
  \left(1+  | \xi  |^2 \right)^{-\LLL},& j<4.
\end{cases}
$$
We are now in position to apply Lemma \ref{lemma-2}, with $\Ps_j$ there defined here
as $\Ps_j: = 
(M_{\nu})^K \EE_{j,M_{\nu}}$.
This means
that  $M_{\nu}^{-K} \psi_{\gamma} = E_{\gamma,M_{\nu}}$
  is the error that we obtain when approximating  the curvelet
$\gamma$
using the budget $M_\nu$.
From Lemma \ref{lemma-2}, we conclude that there is a constant $\varTheta$
(which depends on $C$ and $K$, but not on $M_{\nu}$),
so that the complete family of
normalized error functions 
$\bigl( (M_{\nu})^K {E}_{\gamma,M_{\nu}}\bigr)_{\gamma\in\Dcurv}$ 
forms a Bessel system
with Bessel bound $\varTheta$.

We can now finish estimating $I_+$ using this uniform (i.e., budget-independent)
Bessel property.
Namely,
\begin{equation}\label{principal_estimate}
I_+ 
\le
 \sum_{\nu =0}^{\lfloor \log_2 N\rfloor }
  \left\|
     \sum_{n\in \M_{\nu}} {\omega_{n}}\bigl(M_\nu\bigr)^{-K} \bigl(M_\nu\bigr)^{K} {E}_{\gamma_n,M_{\nu}}
  \right\|_{L_2} 
  \le
  \varTheta
    \sum_{\nu = 0}^{\lfloor \log_2 N\rfloor } 
     \sqrt{ \sum_{n\in \M_{\nu}}
     \bigl(M_{\nu}\bigr)^{-2K}
    | \omega_{n}|^2}.
\end{equation}  
This means we estimate $I_+\le \varTheta \sum _{\nu = 0}^{\lfloor \log_2 N\rfloor} I_{\nu}$, where for  each $\nu$, 
$I_\nu^2 := { \sum_{n\in \M_{\nu}}    \bigl(M_\nu \bigr)^{-2K}  |\omega_{n}|^2}$. 

The contribution to the error at  level $\nu$  can be estimated as
 $I_{\nu}^2
     \le
     (\# \M_{\nu})  \max_{n\in \M_{\nu}}  M_{\nu}^{-2K}
|\omega_n|^2$. 
We make the following elementary observations:
\begin{itemize}
\item Because $\M_{\nu} =(2^{\nu-1},2^{\nu}]$ we have
$\# \M_{\nu} \le 2^{\nu-1}.$
\item The definition of $M_{\nu}$ as $ \lfloor \cost(2^{\nu}) \rfloor=\lfloor C_{\beta} (2^{-\nu}N)^{\beta}\rfloor$
guarantees that $ M_{\nu} \ge \frac12   C_{\beta} (2^{-\nu}N)^{\beta}$. 
In turn, this gives
$M_{\nu}^{-2K} \le C N^{-2\rho} 2^{2\nu \rho}$ (with $C = (2/C_{\beta})^{2K}$).
\item The coefficients indexed by $\M_{\nu}$ satisfy
$
|\omega_n|
\le |\omega_{2^{\nu-1}}|$,
so
$
|\omega_n|^2
\le
2^{2\alpha} \curvenorm{f}^2  2^{-2\nu\alpha }
 $.
 \end{itemize}
It follows that the contribution from the $\nu$th investment level can be estimated as
$$
I_{\nu}^2 
     \le
     (\# \M_{\nu})  \max_{n\in \M_{\nu}}  M_{\nu}^{-2K}
|\omega_n|^2
     \le 
   C  \curvenorm{f}^2   N^{-2\rho} 2^{\nu (1+2\rho -2 \alpha )}.
$$
Taking the square root and summing over all $\nu$, we obtain
$$
I_+ 
\le 
C  \curvenorm{f} N^{-\rho}   \sum_{\nu=0}^{\log m^*} 2^{\nu (\frac12+\rho -\alpha)}
\le 
 C  \curvenorm{f} N^{-\rho}    (m^*)^{\frac12+\rho-\alpha}
\le  
C  \curvenorm{f} N^{\frac12-\alpha}.
$$ 
The second inequality follows because $\rho>\alpha-1/2$, while the last one uses (\ref{terminal_index}).
This together with the estimate of the tail $I_-$ yields the error estimate in Theorem \ref{thm-main}.
\end{proof}
%%%%%%%%%%%%%%%%
%
%%%%%%%%%%%%%%%%

%%%%%%%%%%%%%%%%%%%%%%%%%%%%%%
%%%%%%%%%%%%%%%%%%%%%%%%%%%%%%
%
%Section 3.1: Cartoon Subsection
%
%%%%%%%%%%%%%%%%%%%%%%%%%%%%%%
%%%%%%%%%%%%%%%%%%%%%%%%%%%%%%
\subsection{Approximation of cartoon functions}\label{sec=cartoon}
If $f= \sum_{n=1}^{\infty} \omega_n \gamma_n$ is in the cartoon class, 
then $|\omega_n |\le C n^{-3/2}|\log n|^{3/2}$, and 
simple thresholding gives the $N$-term curvelet approximation rate of $N^{-1} (\log N)^{3/2}$.

If we apply the  Gaussian approximation scheme $\T_N$ with cost function (\ref{cost})
and with $\rho := K \beta > 1$, then we have the following error analysis:
the tail (obtained by thresholding coefficients at $m^*\sim N$) 
can be estimated by the $N$-term thresholding result
\begin{equation*}
I_-     
\le
 C N^{-1} |\log N|^{3/2}.
\end{equation*}
The principal part
is again estimated with the help of (\ref{principal_estimate}): 
namely $I_+ \le \varTheta  \sum_{\nu=1}^{\lfloor \log_2 N\rfloor} I_{\nu}$.
As before, 
$I_{\nu}^2=  \sum_{n\in \M_{\nu}} 
     M_{\nu}^{-2K}
     |\omega_{n}|^2$  
is estimated by controlling 
$\# \M_{\nu}$ and $\max_{n\in \M_{\nu}} M_{\nu}^{-2K} |\omega_n|^2$. 

Again, $\# \M_{\nu} =2^{{\nu}-1}$, $M_{\nu}^{-2K} \le C N^{-2\rho} 2^{2\nu \rho} $ and
$$
|\omega_n|
\le |\omega_{2^{\nu}}|
\quad \Rightarrow
\quad
|\omega_n|^2
\le 
C 2^{-3\nu} | \nu|^{3}.
$$
From this, we have
$$
I_{\nu}^2 
\le 
\# \M_{\nu}  \max_{n\in \M_{\nu}} ( M_{\nu}^{-2K} |\omega_n|^2 )
\le 
C N^{-2\rho}  2^{{\nu(2\rho-2 )}}  | \nu|^3.
$$
Because $\rho>1$, it follows that 
$ I_+ 
\le 
\varTheta 
\sum_{\nu=1}^{\log_2 m^*} I_{\nu}
\le C N^{-\rho} (m^*)^{\rho-1} |\log m^*|^{3/2} 
\le CN^{-1} |\log N|^{3/2}.$
Adding the estimate for $I_-$ gives
$$\|f - \T_N f\|_2\le I_- + I_+ \le C N^{-1} |\log N|^{3/2}.$$

\smallskip
\begin{remark}
We note that the 
sparsity condition $|\omega_n |\le C n^{-3/2}|\log n|^{3/2}$
does not characterize the cartoon functions.
These latter functions satisfy additional conditions, for example
each cartoon $f$ is bounded; this implies that, for such $f$,
any coefficient from level $j$
satisfies 
$$|\omega_{\gamma} |= |\langle f, \gamma\rangle| \le \|f\|_{\infty} \|\gamma\|_1 \sim 2^{-\frac{3j}{4}} \|f\|_{\infty}.$$
Thus for  coefficients $\omega_{\gamma}$  from dyadic level $j\ge  2\log N$, we have
$$|\omega_{\gamma}| \lesssim 2^{-3j/4} \le N^{-3/2} \le N^{-3/2}  (\log N)^{3/2} \le (m^*)^{-3/2} (\log m^*)^{3/2},$$
and, hence, $\gamma$ belongs to the tail (and will not be approximated). 
%Consequently, 
As pointed out in \cite[p.242]{CD}, the rearrangement of the curvelet expansion 
 may be done relatively efficiently; 
the large coefficients can be found in a relatively shallow tree of depth $\mathcal{O}( \log N)$. \end{remark}

%%%%%%%%%%%%%%%%%%%%%%%%%%%%%%
%%%%%%%%%%%%%%%%%%%%%%%%%%%%%%
%
%Section 4: Curvelet Approximation
%
%%%%%%%%%%%%%%%%%%%%%%%%%%%%%%
%%%%%%%%%%%%%%%%%%%%%%%%%%%%%%
\section{Individual curvelet approximation}
\label{sec:individualerror}

We prove in this section Lemma \ref{lemma-1}, which treats the error in approximating the curvelet generators. 
The curvelet generators are band-limited Schwartz functions, satisfying
 (for some constants $0<A<B<\infty$) that for all $j\in\Nn$,
 $\mathrm{supp}(\widehat{\gam^{(j)}}) \subset [-B,B]^2$ 
 and  that, for $j\ge 4$,
 $\gam^{(j)}(\xi_1,\xi_2) =0$ whenever $|\xi_1|<A$ (this is Lemma \ref{support}).
We obtain an $M$-term Gaussian approximation by employing a truncated
``semi-discrete convolution'' 
as considered and analyzed in \cite{HR}.
However, applying directly  those results
is insufficient for our needs here:
the error estimates in the frequency domain  asserted in
Lemma \ref{lemma-1} require the Gaussian approximant 
to satisfy analytic conditions beyond those that can be deduced
from the analysis in \cite{HR}.
In particular, the error must have algebraic decay in frequency 
and, for $j\ge 1$, a ``directional vanishing moment'';
cf.\ (\ref{one_term_rotation_error_mod}).
The latter condition means that 
the function $(\xi_1,\xi_2)\mapsto \widehat{\EE_{j,M}}(\xi_1,\xi_2)$
must have a zero of prescribed order along the line $\xi_1=0$.

Although the main objective of this section is the proof of Lemma \ref{lemma-1}, which treats the approximation
of the curvelet generators (a bivariate result), it is natural to do so by 
studying a slightly more general problem: semi-discrete approximation of band-limited functions in $\Rr^d$.

%
%{\bf comment amos: more edits here}%
\paragraph{Band-limited Schwartz functions}
If $\Gamma\subset \Rr^d$ is compact,
consider the  subspace $\B_{\Gamma}\subset \S(\Rr^d)$
$$\B_{\Gamma}:=\{f\in \S(\Rr^d)\mid \supp{\widehat{f}}\subset \Gamma\}$$ 
(this is simply the image of $C_{\Gamma}^{\infty}(\Rr^d)$
under the inverse Fourier transform).
This is a closed ideal with respect to convolution.
The topology of this subspace can be defined via the family of
seminorms 
 $\varrho_k(f):= \max_{|\gamma|\le k} \| D^{\gamma} \widehat{f}\|_{L_{\infty}(\Rr^d)}$.

\paragraph{Fourier multipliers on $\B_{\Gamma}$}
If $\tau\in \S(\Rr^d)$, then $\tau*:\S(\Rr^d)\to \S(\Rr^d):f\mapsto \tau*f$ 
is continuous.
 This idea can be generalized in a number of ways. We make use of the following, which can be 
 proved by standard arguments:
%%%%%
%
%%%%%
\begin{lemma}\label{deconvolution}
If $U$ is a neighborhood of $\Gamma$ and $\widehat{\tau}\in C^{\infty}(U)$, 
then $\tau*:\B_{\Gamma}\to \B_{\Gamma}:f\mapsto \tau*f$ is continuous; indeed for every $K$
there is $C$ so that  for all $f\in \B_{\Gamma}$,
$\varrho_K(\tau*f)\le C  \varrho_K(f)$.
Furthermore, if $\widehat{\tau}$ is nonzero on $\Gamma$, 
then $\tau*:\B_{\Gamma}\to \B_{\Gamma}$ 
is a Fr{\' e}chet space automorphism, and  for every $K$ there is $C<\infty$ 
so that for all $f\in \B_{\Gamma}$, $\varrho_K(f)\le C\varrho_K(\tau*f)$.
\end{lemma}
%%%%%
%
%%%%%

%%%%%
%
%%%%%
%%%%%%%

%
%
\paragraph{The semi-discrete convolution}
Consider the functions band-limited in a ball $\Gamma = B_r(0)$, writing
$\H_r:=\B_{B_r(0)}$ for short.
For $F\in\H_r$, Lemma \ref{deconvolution} guarantees existence of $f_{F}\in \H_r$ 
so that Gaussian deconvolution $F=\phi*f_F $ holds. 
(Recall that $\phi(x) =e^{-|x|^2}$.)
Furthermore, for every $K$, there exist constants $0<c<C<\infty$ so that for all $F\in \H_r$,
$c \varrho_K(F) \le \varrho_K(f_F) \le C \varrho_K(F)$.

We approximate $F$ by the semi-discrete convolution $\ful F$, defined as
\begin{equation}\label{full_scheme}
\ful F :=
(2\pi)^{-d} h^d \sum_{\alpha \in h\Zz^d}  f_{F}(\alpha) \phi(\cdot - \alpha).
\end{equation}
From \cite[Proposition 1]{HR}, there exist constants $C$ and $c>0$ so that for $h<\pi/r$ and for
all $F\in \H_r$,
$$
\|F-\ful F\|_{\infty} \le C \|\widehat{F}\|_{L_1(\Rr^d)}
h^d
e^{\mbox{}-\frac{c}{h^2}}
$$
holds.
The constants $c$ and $C$ are global: they depend on $r$, but  are independent of $h$ and $F\in \H_r$.

In section \ref{sec=simple_finite}, we truncate the infinite series (\ref{full_scheme}), 
and study its approximation properties for functions in $\H_r$.
This results in an approximation scheme generated by an operator $\notful$, 
a straightforward modification of the scheme considered in \cite{HR} 
which enjoys rapid decay of the error in frequency.

In section \ref{sec=vm_finite}, we work with functions band-limited in a square, 
but vanishing in a tubular neighborhood of $\Rr^{d-1}\times \{0\}$ (the $\xi_2$-axis when $d=2$). 
In this case, we set $\Gamma  = \{\xi\in [-B,B]^d\mid |\xi_1|\ge A\}$,
and write 
$$
\H_{A,B} := \B_{\Gamma} :=
\left\{f\in \S(\Rr^d)\mid \supp{\widehat{f}\;}\subset [-B,B]^d,\ \widehat{f}(\xi)=0 \text{ if }|\xi_1|\le A\right\}.
$$
%

%%%%%%%%%%%%%%%%%%%%%%%%%%%%%%
%%%%%%%%%%%%%%%%%%%%%%%%%%%%%%
%
%Section 4: First Subsection
%
%%%%%%%%%%%%%%%%%%%%%%%%%%%%%%
%%%%%%%%%%%%%%%%%%%%%%%%%%%%%%
\subsection{The scheme $\notful$: approximation using finitely many centers}\label{sec=simple_finite}

In order to get a finite linear combination of Gaussians,
we approximate $F$ by $\notful F$,
where
\begin{equation}
\notful  F:= (2\pi)^{-d} h^d \sum_{\alpha\in h\Zz^d }\sigma(h \alpha) f_{F}(\alpha)\phi(\cdot-\alpha),
\end{equation}
with $f_{F}$ and $\phi$ as in the previous subsection,   
and
$\sigma:\Rr^d\to [0,1]$ is a smooth cut-off function equaling $1$ in the unit ball and supported in  
ball of radius 2.

There are a few points to make.
First, the approximation properties of $\ful F$ are inherited by $\notful F$;
removing from the sum all shifts outside  $B_{{2}{h^{-1}}}(0)$ has  little effect because 
of the rapid decay of the coefficients and the uniform boundedness of $\phi$.
Second, the smooth truncation ensures we preserve, to some extent, 
the localization
of the Fourier transform of $\ful F$.
Finally, if $h<1$,
the number of shifts 
$n:=n(h) = \# ( h\Zz^d \cap B_{{2}{h^{-1}}}(0))$ 
that are being used for a given value of
$h\le 1$ satisfies 
\begin{equation}\label{cardinality_estimate}
a_d h^{-2d}\le n\le b_d h^{-2d}\end{equation}
for some $d$-dependent\footnote{
The value of $n$ is equal to the number of integer lattice points in the interior of  a ball of radius $2/h^2$.
For $d=1$, (\ref{cardinality_estimate}) holds with $b_1=4$ and for $d=2$, 
(\ref{cardinality_estimate}) holds with $b_2=36$.
} 
constants $a_d$ and $b_d$. 
%

%%%%%%%%%%%%%%%%%%%%%%%%%%%%%%%
%
%
%
%
%
%%%%%%%%%%%%%%%%%%%%%%%%%%%%%%%
\paragraph{Fourier transform of ${\notful  F}(x)$}  \label{truncated_FT}
%%%%%%%%%%%%%%%%%%%%%%%%%%%%%%%
Apply the Poisson summation formula to obtain
\begin{eqnarray}
\widehat{\notful  F}(\xi)
&:=&
h^d\widehat{\phi}(\xi)(2\pi)^{-d}
  \sum_{\alpha\in h\Zz^d}\sigma(h \alpha) f_{F}(\alpha)e^{-i\langle\alpha, \xi\rangle} \nonumber\\
&=&
\widehat{\phi}(\xi)(2\pi)^{-d}
  \sum_{\beta \in 2\pi\Zz^d/h}\bigl(\sigma(h \cdot ) f_{F}\bigr)^{\wedge}(\xi+\beta) \nonumber\\
&=&
\widehat{\phi}(\xi)(2\pi)^{-d}
  \sum_{\beta \in 2\pi\Zz^d/h}\left(\widehat{\sigma(h\cdot)}* 
  \frac{\widehat{F}}{\widehat{\phi}}\right)(\xi+\beta).
\label{notful_fourier_formula}
\end{eqnarray}
From this, we have the following result.
%%%%%%%%%%%
%
%%%%%%%%%%%
\begin{lemma}\label{Flat_fourier_decay}
For each $K>0$ there is $C>0$ so that for all $F\in \H_r$, %and all $h< \frac{\pi}{r}$
$$
|\widehat{F} (\xi) - \widehat{\notful  F}(\xi)|
\le
C\varrho_{K}(F) h^K \left(1+|\xi| \right)^{-K} .
$$
\end{lemma}
%%%%%%%%%
%%%%%%%%%
\begin{proof}
We   estimate the Fourier transform of the above error expression. By  (\ref{notful_fourier_formula}) 
we have
\begin{eqnarray}
|\widehat{F}(\xi) - \widehat{\notful  F}(\xi)|
&\le&
|\widehat{\phi} (\xi)|
\left|
  \widehat{f_{F}}(\xi)
  -
  \bigl(\sigma(h \cdot) f_{F}\bigr)^{\wedge}(\xi)
\right|\nonumber\\
&&
+\left|\widehat{\phi}(\xi)\right|
\sum_{\substack{\beta \in 2\pi\Zz^d/h\\\beta\ne 0}}
  |\bigl(\sigma(h \cdot ) f_{F}\bigr)^{\wedge}(\xi+\beta)|.\label{FT_error}
\end{eqnarray}
Note that 
$\widehat{\sigma(h\cdot)}(\xi) = h^{-d}\widehat{\sigma} (\xi/h)$, and
 $\bigl(\sigma(h \cdot) f_{F}\bigr)^{\wedge}(\xi) = h^{-d} f_F* \widehat{\sigma} (\cdot/h)$.
 So by standard arguments we have 
$
\left\|\widehat{f_{F}} -   \bigl(\sigma(h \cdot) f_{F}(\cdot)\bigr)^{\wedge}  \right\|_{\infty} 
\le 
C\varrho_K(f_{F}) h^K 
$, which
we apply to the first term in (\ref{FT_error}).
Thus, we obtain, for all $\xi\in \Rr^2$,
$$
|
\widehat{\phi}(\xi)|\left|\widehat{f_{F}}(\xi) -
  \bigl(\sigma(h \cdot) f_{F}\bigr)^{\wedge}(\xi)
\right| 
\le 
C\varrho_K(f_F) h^K
 |\widehat{\phi}(\xi)|
 \le
 C \varrho_K(F) h^K
 (1+|\xi|)^{-K},
$$
since $\varrho_K(f_F) \le C \varrho_K(F)$ and $\widehat{\phi}$ 
is rapidly decaying.

As a Schwartz function, $\sigma$ satisfies $|\widehat{\sigma}(\xi)|\le C_K (1+|\xi|)^{-K}$
for all $K>0$.
It follows that
$|\widehat{\sigma(h\cdot)}(\xi) |\le C_K h^{-d}(1+\frac{|\xi|}{h})^{-K}$.
Thus, for each $K>0$ there is a constant $C$ so that for all $\xi\in\Rr^2$
\begin{equation}\label{square_fourier_estimate}
\left| \bigl(\sigma(h \cdot) f_{F}\bigr)^{\wedge}(\xi)\right|
 =
\left|
\left(
\widehat{\sigma(h\cdot)}* \widehat{f_{F}}
\right)(\xi)\right| 
\le
C
\|\widehat{f_{F}}\|_{\infty}
\left(1 +
\frac
{\dist(z,B_r)}
{h}
\right)^{-K},
\end{equation}
since $\dist(\xi,B_r) \le \dist(\xi, \supp{\widehat{f_F}})$.
This estimate implies  the uniform bound
$$
\sum_{\beta \in 2\pi\Zz^d/h}
  \left({\sigma(h\cdot)f_F}\right)^{\wedge}(\xi+\beta)
  \le
  C
  \|\widehat{F}\|_{\infty}.
$$
 For $|\xi|\ge \pi/h$, the Gaussian satisfies
 $|\widehat{\phi}(\xi)| \le C_K h^K (1+|\xi|)^{-K}$, so 
$$
 \left|\widehat{\phi}(\xi)\right|
\sum_{\substack{\beta \in 2\pi\Zz^d/h\\\beta\ne 0}}
  |\bigl(\sigma(h \cdot ) f_{F}\bigr)^{\wedge}(\xi+\beta)|
  \le 
  C_K  \|\widehat{F}\|_{\infty} h^K (1+|\xi|)^{-K}.
$$
 For $|\xi| \le {\pi/h}$,
 the sum over $2\pi \Zz^d\setminus \{0\}$ can be estimated by  (\ref{square_fourier_estimate}) 
 and working with annuli:
\begin{equation}\label{off_center_fourier_estimate}
\sum_{\substack{\beta \in 2\pi\Zz^d/h\\ \beta \ne 0}}
\left(1 +\frac{ \dist(\xi+\beta,B_r)  }{h}  \right)^{-K}
\le
C \sum_{m=1}^\infty m^{d-1}\left(1+\frac{m \pi}{h^2}\right)^{-K}%\nonumber\\
\le
C_K h^{2K}.
\end{equation}
Thus, for every $\xi$,
$$
\left|\widehat{\phi}(\xi)\right|
\sum_{\substack{\beta \in 2\pi\Zz^d/h\\\beta\ne 0}}
    \left|\bigl(\sigma(h \cdot) f_{F}\bigr)^{\wedge}(\xi+\beta)\right|
\le
C  \|\widehat{F}\|_{\infty}
h^K  \left(1+|\xi|\right)^{-K}
$$
holds and the lemma follows.
\end{proof}

%%%%%%%%%%%%%%%%%%%%%%%%%%%%%%
%%%%%%%%%%%%%%%%%%%%%%%%%%%%%%
%
%Section 4: Second Subsection
%
%%%%%%%%%%%%%%%%%%%%%%%%%%%%%%
%%%%%%%%%%%%%%%%%%%%%%%%%%%%%%
\subsection{Recovering  vanishing moments}\label{sec=vm_finite}

In this section, we recover the vanishing moments of a Gaussian approximant. 
For the purposes of Lemma  \ref{lemma-1}, we would pick $J=2$, 
but having more vanishing moments is not a challenge. 
The results in this section show that vanishing moments of any finite order can be obtained.

Let $\varDelta:\mathcal{S}(\Rr^d)\to \mathcal{S}(\Rr^d)$ 
be the $J$-fold symmetric difference in the horizontal direction whose
symbol is
$$\Xi(\xi) := \left( \sin (\xi_1/\kappa)\right)^J$$
with $\kappa>0$.
Thus, $\varDelta f =\sum_{k=0}^J c_k f(\cdot+(2k-J)\vec{e_1}/\kappa )$ is a combination of $J+1$ 
shifts of $f$ along the first coordinate, with spacing $2/\kappa$. 

There is a $(J,\kappa)$-dependent constant $C$ so that  
$\rho_K(\varDelta f) \le C \rho_K(f)$, for all $f\in \S(\Rr^d)$.
On the other hand, for sufficiently large $\kappa$,
Lemma \ref{deconvolution} guarantees that the operator $\varDelta$ is a Fr{\'e}chet space automorphism of 
$\H_{A,B}\subset \mathcal{S}(\Rr^d)$.
(We note that  $\Xi(\xi)$ vanishes when $\xi_1=0$, but this does not pose a problem, since $A > 0$.)

We construct a Gaussian approximant on $\H_{A,B}$ as follows: 
for $F\in \H_{A,B}$, let $F_0 =\varDelta^{-1}F\in \H_{A,B}$ 
(so $F= \varDelta F_0$).
Then
$$T_h F :=\varDelta \notful F_0.$$
In other words, $T_h = \varDelta \notful \varDelta^{-1}$ is obtained by conjugating by $\varDelta$.
Because $\notful$ is a combination of $n(h) = \# \{\alpha\in h\Zz^d \mid |\alpha|\le 2/h\}$ 
translates of $\phi$, $T_h F$ uses at most 
$(J+1)\times n(h)$ elements. 

Regarding the decay of the approximant in the spatial as well as in the frequency domain, 
we get the following result.

\begin{lemma} \label{lem-1a}
For every $J$ and $K\in \Nn$ there exists a constant $C>0$ 
such that  for all $h > 0$ the following inequality holds for all $F\in \H_{A,B}$ and all $\xi \in \Rr^2$:
\begin{align}
|\widehat{F}(\xi) -  \widehat{T_h F}(\xi)| 
&\leq 
C\varrho_{K}(F) h^K \min\left(1,|\xi_1|^{J}\right) (1+|\xi|)^{-K}
\label{eq:errorfrequency}.
\end{align}
\end{lemma}
\begin{proof} 
We have
$|\widehat{F}(\xi) -  \widehat{T_h F}(\xi)| = |\Xi(\xi)| |\widehat{F_0}(\xi) - \widehat{\notful F_0}(\xi)|$.
Since $ |\Xi(\xi)|\le C\min\left(1,|\xi_1|^{J}\right),$  Lemma \ref{Flat_fourier_decay} ensures that
$$
|\widehat{F}(\xi)-  \widehat{T_h F}(\xi)| \le  C\min\left(1,|\xi_1|^{J}\right)h^K(1+|\xi|)^{-K}\varrho_{K}(F_0) .
$$
 Lemma \ref{deconvolution} gives $\varrho_{K}(F_0) \le C \varrho_{K}(F) $, 
and (\ref{eq:errorfrequency}) follows.
\end{proof}

%%%%%%%%%%%%%%%%%%%%%%%%%%%%%%
%%%%%%%%%%%%%%%%%%%%%%%%%%%%%%
%
%Section 4: Third Subsection
%
%%%%%%%%%%%%%%%%%%%%%%%%%%%%%%
%%%%%%%%%%%%%%%%%%%%%%%%%%%%%%
\subsection{Proof of Lemma \ref{lemma-1}}
\begin{proof}
For the case $j<4$, select $h$ so that $b_2 h^{-4}=M$, where $b_d$ is the constant in (\ref{cardinality_estimate}).
Then $\G_{j, M}:=\notful \gam^{(j)}$ is a linear combination of fewer than $M$ Gaussians.  
Applying Lemma \ref{Flat_fourier_decay} in conjunction with Lemma \ref{schwartz_bound} gives 
$|\widehat{\EE_{j,M}}(\xi)| 
\le 
C\varrho_{4K}(\gam^{(j)}) h^{4K} \left(1+|\xi| \right)^{-4K}
\le 
C_K  
M^{-K} 
\left(  1+  |\xi|\right)^{-K}$.

For $j\ge4$, we select $h$ so that $b_2 h^{-4}=M/(J+1)$. In that case, we use the operator $T_h$ introduced
in section \ref{sec=vm_finite}, with $J=2$. 
The approximant $\G_{j, M}:=T_h \gam^{(j)}$ uses
at most $M$ points.
Inequality (\ref{one_term_rotation_error_mod}) 
follows from Lemma \ref{lem-1a} combined with the uniform bound in Lemma \ref{schwartz_bound}.
\end{proof}
\begin{remark}
We note that the minimal positive investment for a general curvelet %of Gaussians, $N_0$,
by a Gaussian 
approximant with $J = 2$ vanishing moments is $N_0 = J+1 = 3$.
\end{remark}

%%%%%%%%%%%%%%%%%%%%%%%%%%%%%%
%%%%%%%%%%%%%%%%%%%%%%%%%%%%%%
%
%Section 5
%
%%%%%%%%%%%%%%%%%%%%%%%%%%%%%%
%%%%%%%%%%%%%%%%%%%%%%%%%%%%%%
\section{The Bessel property} %- Proof of Lemma \ref{lemma-2}} 
\label{sec:bessel}

As a last remaining task, we have to prove Lemma \ref{lemma-2}. 
In this section, we prove a slight generalization (from which the lemma follows). 
Namely, that
\begin{equation} \label{eq:systemoferrors}
  \Derr = 
  \left\{{\psi}_{\gamma} := U_{\gamma} \Ps_{j(\gamma)}\mid \gamma \in \Dcurv, 
  \gamma=U_{\gamma} \gam^{(j(\gamma))}\right\}
\end{equation}
is a Bessel system when the generators $ \Ps_{j}\in L_2(\Rr^2)$ satisfy the bounds
\begin{align}
|\widehat{\Ps_{j}}(\xi)|&\le C \min(|\xi_1|^J,1) (1+|\xi|)^{-K}, & j \geq 4, \label{eq:errorboundJK}  \\
|\widehat{\Ps_{j}}(x)|&\le C (1+|\xi|)^{-K}, & j <4,\label{eq:errorbound_space} 
\end{align}
with $J>1$ and $K>2$  (and $C<\infty$ independent of $j$).

Throughout this section, we frequently make use of the convenient notation $\psi_{j,\ell,k}$ to denote 
$\psi_{\gamma}\in \Derr$ 
when
$\mu(\gamma) = (j,\ell,k)$.

\subsection{Generalized shift invariant systems}
We first review some pertinent details about shift-invariant (SI) systems. 
Our sources in this regard are \cite{RS}, \cite{LWW} and \cite{CR}.
A GSI (generalized SI) system, defined, say,
on $\Rr^2$, is generated by a (finite or countable) set $(\varphi_i)_{i\in I}$ of 
$L_2(\Rr^2)$-functions, each associated with a 2D lattice $\Upsilon_i$. 
The GSI system is 
$$X:=\{\varphi_i(\cdot-k),\quad i\in I,\ k\in \Upsilon_i\}.$$
In order to analyze the Bessel (or another related) property of the system,
one associates it with its  {\it dual Gramian kernel}. 
To this end, recall that the {\it dual lattice} $\hat{\Upsilon}$  of a lattice  $\Upsilon$ is defined as
$$
\hat{\Upsilon}:=\{k \in \Rr^2: \langle k,\Upsilon \rangle \subset 2\pi \Zz\}.
$$

\begin{definition} \label{def:dualgramian} 
The dual Gramian kernel $\fiber_{X}$ is
(formally) defined on $\Rr^2\times \Rr^2$ as follows:
$$\fiber_{X}(\xi,\tau):=\sum_{i\in I} \delta_i(\xi-\tau) \frac{\hat{\varphi}_i(\xi)
\overline{\hat{\varphi}_i(\tau)}}{|\Upsilon_i|},$$
with $\delta_i(\xi)=1$ on $\hat{\Upsilon}_i$, and $=0$ otherwise. 
$|\Upsilon_i|$ denotes the area of the fundamental domain of $\Upsilon_i$.
\end{definition}

The theory of SI spaces is based on {\it fiberization}:
the art of associating the system $X$ with constant coefficient Hermitian matrices (`fibers') 
whose entries are sampled from  the dual Gramian kernel, and characterizing properties of $X$ via
the uniform satisfaction of the analogous properties over (almost) all the fibers. 
A row in a given fiber corresponds to some $\xi\in \Rr^2$, and a column to  $\tau\in \Rr^2$, 
with the $(\xi,\tau)$-entry being $\fiber_{X}(\xi,\tau)$. 
Therefore, the $\ell_1$-norm of the $\xi$-row of any fiber is bounded by
\begin{equation*} 
\fiber_{X,\ast}(\xi):=\sum_{\tau\in \Rr^2}|\fiber_{X}(\xi,\tau)|.
\end{equation*}
(Note that $\fiber_{X}(\xi,\cdot)$ may assume non-zero values only on 
$\bigcup_{i\in I} (\xi+\hat{\Upsilon}_i)$
which is a countable set.) 
Therefore, the fibers must be (essentially) uniformly bounded, whenever
$$\| \fiber_{X,\ast}\|_{\infty} <\infty:$$
the above boundedness implies that fibers are uniformly bounded in $\ell_\infty$, 
and since the fibers are Hermitian, 
it follows then that they are uniformly bounded in the requisite $\ell_2$-norm. 
The GSI theory then concludes that the  system $X$ is Bessel.

The above argument applies whenever the fiberization technique is available;
however, not every GSI system is `fiberizable', \cite{RS}. On the other hand, the condition 
\begin{equation} \label{eq:l1boundfibers}
\| \absfiber_{X,\ast}\|_\infty<\infty 
\quad \text{for the kernel} \quad 
\absfiber_{X}(\xi,\tau):=\sum_{i\in I} \delta_i(\xi-\tau) \frac{|\hat{\varphi}_i(\xi)|
|\hat{\varphi}_i(\tau)|}{|\Upsilon_i|},
\end{equation}
implies the Bessel property of the GSI system {\it unconditionally}, 
as  the following result makes clear. 
It is taken from Theorem 3.1 of \cite{CR}. 
The proof of this result is essentially due to \cite{LWW} (cf.\ Theorem 3.4 there and its proof).

\begin{result} \label{res:bessel} 
Let $X$ be a GSI system as above, associated with dual  Gramian kernel $\fiber_{X}$. 
If $\absfiber_{X,\ast}$ 
is essentially uniformly bounded, then $X$ is a Bessel system with
Bessel constant $\Theta = \| \absfiber_{X,\ast}\|_\infty.$
\end{result}

\subsection{$\mathcal{D}_{\Ps}$ 
as a generalized shift invariant system}

Denote
$X= \Derr $
and observe that it is a GSI system, whose GSI generators
are the functions 
$\psi_\gamma$ with   $k(\gamma)=0.$
Thus, the index set of the GSI generators is
$$
I
:=
\{i =(j,\ell)\in \Nn^2\mid   0\le  \ell \le 2^{\lfloor j/2\rfloor}-1\}.
$$

A basic member in the system is then given by
$\psi_{\gamma}(x) = {|D_j|}^{1/2} \Ps_j \bigl(D_{j} R_{j,\ell}^*( x - R_{j,\ell}k) \bigr)$, 
so
the translations of a GSI generator for a given index $i=(j,\ell)\in I$ 
comprise a rotated version of the original curvelet grid $\grid_j$. 
More precisely, the GSI shifts for $i=(j,\ell) \in I$ are given on the grid 
$$\Upsilon_i := R_{j,\ell}\grid_j= 2 \pi R_{j,\ell} (D_\delt^{(j)})^{-1} D_{-j} \Zz^2,
$$
 based on the parabolic scaling matrix $D_{-j}$, 
 the rotation $R_{j,\ell}$ and the correction matrix (cf.\ \eqref{eq:grid})
\begin{equation}\label{eq:correctionmatrix}
D_\delt^{(j)} 
= 
\begin{pmatrix}
             \delt_1(j) & 0 \\ 0 &  \delt_2(j)
\end{pmatrix}.
\end{equation}
Thus, the dual lattice is then
\begin{equation}\label{eq:hatUpasilonis}
\hat{\Upsilon}_i=R_{j,\ell}D_\delt^{(j)}D_j\Zz^2.
\end{equation}
The area of the fundamental 
cell of the grid $\Upsilon_i$ is given by
\begin{equation} \label{eq:fundamentalcellarea} 
|\Upsilon_i| =  |\grid_j| = (2 \pi)^2 |(D_\delt^{(j)})^{-1}| |D_{-j}| = \frac{(2 \pi)^2}{\delt_1(j) \delt_2(j) }|D_{-j}|\ge C |D_{-j}|,
\end{equation}
where we have used the uniform boundedness away from $0$ of the correction factors
(see Remark \ref{grid_correction_remark} 
in the appendix). 
That means that, since our goal is to prove that 
$ \absfiber_{X,\ast}$ is essentially bounded,
(so that we can invoke Result \ref{res:bessel}),
we may modify, without loss, the kernel
$\absfiber_{X}$ (from \ref{eq:l1boundfibers}) to be
$$\absfiber_{X}(\xi,\tau):=\sum_{i=(j,\ell)\in I} \delta_i(\xi-\tau) 
|D_j| |\hat{\varphi}_i(\xi)| |\hat{\varphi}_i(\tau)|,$$
as we do in the rest of this section. Here, as before, $j=i(1)$ and $\ell = i(2)$, 
i.e., 
the dilation parameter $j$ and the rotation $\ell$ are encoded in the index $i=(j,\ell)$.
Then,
\begin{equation} \label{eq:FourierErrorFunctiona}
\hat{\varphi}_i(\xi) 
=
\widehat{\psi}_{(j,\ell,0)} (\xi) = \widehat{U_{(j,\ell,0)}\Ps_{j}} (\xi)
= 
 |D_{j}|^{-1/2} \widehat{\Ps}_j (D_{-j} R_{j,\ell}^* \xi). 
\end{equation}
Therefore,
$$\absfiber_{X}(\xi,\tau)=\sum_{(j,\ell)\in I} \delta_{(j,\ell)}(\xi-\tau) 
\bigl| \widehat{\Ps}_j (D_{-j} R_{j,\ell}^* \xi)\bigr| 
\bigl|\widehat{\Ps}_j (D_{-j} R_{j,\ell}^* \tau)\bigr|$$
and we can write
\begin{equation} \label{eq:kernelsimplified} 
\absfiber_{X,\ast}(\xi)=
\sum_{i=(j,\ell)\in I} 
\bigl|\widehat{\Ps}_j (D_{-j} R_{j,\ell}^* \xi)\bigr|\, 
\sum_{\zeta\in \hat{\Upsilon}_i}
\bigl|\widehat{\Ps}_j (D_{-j} R_{j,\ell}^* (\xi-\zeta))\bigr|.
\end{equation}
Thanks to (\ref{eq:hatUpasilonis}), we observe that the second summation here
equals
$$\sum_{k\in \Zz_j}
\bigl|\widehat{\Ps}_j (D_{-j} R_{j,\ell}^* \xi-k)\bigr|,$$
with $\Zz_j:=D_\delt^{(j)}\Zz^2$, hence is bounded by
$$\left\|\sum_{k\in \Zz_j}
\bigl|\widehat{\Ps}_j (\cdot-k)\bigr|\,\right\|_\infty
\le C \sup_{\xi\in \Rr^2} \sum_{k\in \Zz_j}
\big(|1+ |\xi-k|\bigr)^{-K}
.$$ 
This can be estimated as follows: for any $K>2$ 
there is a constant $C_K$ so that for any discrete set $\Xi\subset \Rr^2$,
one has
$$\sup_{\xi\in\Rr^2}\sum_{k\in \Xi}
\big(|1+ |\xi-k|\bigr)^{-K}\le C_K q(\Xi)^{-2},$$ where $q(\Xi):= \min \{|k_1-k_2| : k_1,k_2\in \Xi, k_1\neq k_2\}$
is the minimal separation of $\Xi$.
For $\Xi = \Zz_j$, the estimates (\ref{correction_bounds}) 
ensure that $q(\Zz_j) \ge 4\pi$ for all $j$, and thus
(\ref{eq:kernelsimplified}) leads to the following inequality:
\begin{equation}\label{lem:besselone}
\absfiber_{X,\ast}(\xi)\le C
\sum_{i=(j,\ell)\in I} 
\bigl|\widehat{\Ps}_j (D_{-j} R_{j,\ell}^* \xi)\bigr|. 
\end{equation}

\begin{lemma} \label{prop:Bessel} 
Assuming \eqref{eq:errorboundJK} and 
\eqref{eq:errorbound_space} hold
(for $J>1$ and $K > 2$), the function 
$$\xi\mapsto\sum_{i=(j,\ell)\in I} 
\bigl|\widehat{\Ps}_j (D_{-j} R_{j,\ell}^* \xi)\bigr|$$
is essentially uniformly bounded.
\end{lemma}

%%%%%%%%%%%%%%%
%
%%%%%%%%%%%%%%%

Assuming Lemma \ref{prop:Bessel}, 
we can finish the proof of the Bessel property of $\Derr$.

\begin{proof}[Proof of Lemma \ref{lemma-2}]

The uniform  boundedness of $\absfiber_{X,\ast}$ 
is a consequence of
(\ref{lem:besselone}) and
 the uniform boundedness of 
$\xi\mapsto\sum_{i=(j,\ell)\in I} 
\bigl|\widehat{\Ps}_j (D_{-j} R_{j,\ell}^* \xi)\bigr|,$
which is asserted in Lemma \ref{prop:Bessel}.
Result \ref{res:bessel}  then
guarantees
that $\Derr$ is a Bessel system. 
\end{proof}

\subsection{Proof of Lemma \ref{prop:Bessel}} 

%%%%%%%%%%%%%%%
%
%%%%%%%%%%%%%%%
\begin{proof}[Proof of Lemma \ref{prop:Bessel}]
Note that 
it suffices to estimate  the tail
$\sum_{\substack{(j,\ell)\in I, j\ge 4}} 
\bigl|\widehat{\Ps}_j (D_{-j} R_{j,\ell}^* \xi)\bigr|
$
(since the number of terms we omit in this way is finite, and each is bounded).
The  terms  in this sum are subject to the bound \eqref{eq:errorboundJK}.
Lemma \ref{prop:Bessel} is a consequence of Lemma \ref{lem=star_lemma}, which we prove below.
\end{proof}

We define the functions 
$$\eta_{J,K}(\xi) := \min(|\xi_1|^J,1) (1+|\xi|)^{-K}.$$
These are essentially the bounding functions for the generators $\widehat{\Ps_{j}}(\xi)$, $j \ge 4$ 
introduced in \eqref{eq:errorboundJK}. 
The next lemma applies to  these functions
the star-norm,
 $\|F\|_*:=\sup_{\xi \in \Rr^2} 
\sum_{j} 
\sum_{\ell = 0}^{2^{\lfloor j/2\rfloor}-1} 
  |F(D_{-j} R_{j,\ell}^*\xi)| $, which was introduced in \cite[(2.1)]{dHGR}.
%%%%%%%%%%
%
%%%%%%%%%%
\begin{lemma} \label{lem=star_lemma}
If $J>1$, $K > 2$, and $\eta_{J,K}(\xi)=\min\left(1,|\xi_1|^{J}\right) (1+|\xi|)^{-K}$, then
\[ \|\eta_{J,K}\|_* 
= 
\sup_{\xi \in \Rr^2} 
\sum_{j \in \Nn} 
\sum_{l = 0}^{2^{\lfloor j/2\rfloor}-1} 
  |\eta_{J,K}(D_{-j} R_{j,l}^*\xi)| 
< 
\infty. 
\]
\end{lemma}
%%%%%%%%%%
%
%%%%%%%%%%
 We note that this is a direct adaptation of \cite[Proposition 2.3]{dHGR}, 
 with weaker hypotheses
 on $\eta_{J,K}$. Note also that Lemma 
\ref{lem=star_lemma} implies indeed Lemma \ref{prop:Bessel}.

\begin{proof}[Proof of Lemma \ref{lem=star_lemma}]
Without loss, we can assume $J <K$.
Otherwise, we could replace $J$ by $\tilde{J} = \min(J, 1+K/2)$, and note that
$1<\tilde{J}$, and $\tilde{J} < K$.
Because  $\eta_{J,K} \le \eta_{\tilde{J},K}$,  and the fact that $\|\cdot \|_*$ is increasing 
(meaning that if $F\le G$ then $\|F\|_* \le \|G\|_*$), we have
$\|\eta_{J,K}\|_* \le \|\eta_{\tilde{J},K}\|_*$.

To estimate the star-norm of $\eta_{J,K}$ it is useful to consider circular sectors $V_{s,t}^k$, with $k\in\{0,1,2,3\}$,
$s\in\Zz$ and $t\in \Nn$. For $k=0$, we define
\[ 
V_{s,t}^0
:= 
\left\{ 
  (r \cos \theta, r \sin \theta): 
  \begin{array}{l}
          2^s \leq r \leq 2^{s+1}, \ 0 \leq \theta \leq \pi/2, \\ 
          2^{-t-1} \leq \cos \theta \leq 2^{-t}
   \end{array}
\right\} 
\subset 
\Rr^2
\]
and we use the natural reflections $S_1 (x,y) = (-x,y)$, $S_2(x,y) = (-x,-y)$ and $S_3(x,y) = (x,-y)$
to define $V_{s,t}^k := S_k V_{s,t}^0$ for $k=1,2,3$.

The indicator functions of these sectors satisfy $\| \chi_{V_{s,t}^{k}}\|_* \le C 2^t$ (with $C=9$). This
inequality is a refinement of \cite[Lemma 2.1]{dHGR} and shown in Lemma \ref{lem:sectorinequality} below. 
The monotonicity of the star-norm and the fact that $\eta_{J,K}$ vanishes along the $y$-axis ensures that
$$\|\eta_{J,K} \|_* 
\le 
\left\| \sum_{k=0}^3 \sum_{s \in \Zz} \sum_{t \in \Nn_0} 
  \bigl(\max_{\xi\in V_{s,t}^k} |\eta_{J,K}(\xi)|\bigr)\chi_{V_{s,t}^k}\right\|_*
\le 
\sum_{k=0}^3 \sum_{s \in \Zz} \sum_{t \in \Nn_0}
  \bigl(\max_{\xi\in V_{s,t}^k} |\eta_{J,K}(\xi)|\bigr) \|\chi_{V_{s,t}^k}\|_*,$$
from which the estimate
\[ 
\|\eta_{J,K}\|_* 
\leq 
C \sum_{s \in \Zz} \sum_{t \in \Nn_0} 2^{t} \max_{\xi \in V_{s,t}^0}|\eta_{J,K}(\xi)|
\]
follows.
Based on the particular definition of the function $\eta_{J,K}$, we further have
\[ \max_{\xi \in V_{s,t}^0}|\eta_{J,K}(\xi)| 
\leq 
\begin{cases}
2^{-sK}&s\ge t,\\
(2^{s+1-t})^{J} 2^{-sK}& 0\le s< t,\\
%2^{-K}
(2^{s+1-t})^{J} &s<0.
\end{cases}\]
This, in combination with the inequality above, yields
\[ \|\eta_{J,K}\|_* 
\leq 
 C   
  \sum_{t \in \Nn_0} 2^t 
    \left(
      \sum_{s =t}^{\infty} 
      2^{  - sK  } 
      +2^{-tJ}\sum_{s =0}^{t} 2^{ (J- K) s }  
      + 2^{-tJ}\sum_{s=1}^{\infty} 2^{-sJ}
    \right). 
  \]
Since $J > 1$ and $K > J $, the series converges.
\end{proof}

\begin{lemma} \label{lem:sectorinequality}
The indicator function for the sector $V_{s,t}^k$ satisfies $\| \chi_{V_{s,t}^k}\|_* \le 9 \times 2^t$.
\end{lemma}

Note that the proof of Lemma \ref{lem:sectorinequality} is almost identical to the proof of \cite[Lemma 2.1]{dHGR} 
except for a couple of  minor changes, one of which leads to
the improved bound $2^t$ here compared with
$4^t$ in \cite{dHGR}, 

\begin{proof}
Because $\|f\circ S_k\|_* = \|f\|_*$ for each reflection $S_k, k=1,2,3$ 
(a proof is given in \cite[Lemma 2.2]{dHGR}), 
it suffices to prove the result for the sectors in the first quadrant: $V_{s,t}^0$.
We proceed in the same way as in \cite[Lemma 2.1]{dHGR}. 
The dilation matrix $D_j$ applied to a vector $\xi = (r \cos \theta, r \sin \theta)$, $r > 0$, $0 \leq \theta < \pi/2$, 
in the first quadrant can be described as $D_j \xi = \xi' = (r' \cos \theta', r' \sin \theta')$ with 
the new polar coordinates $(r',\theta')$ determined by
\begin{align*}
r' &= \rho(\cos \theta, j) r \\
\tan \theta' &= 2^{-\lfloor j/2 \rfloor} \tan \theta.
\end{align*}
Here, the function $\rho: (0,1] \times \mathbb{N} \to \Rr_+$
is given by
\[\rho(\alpha,j) := \sqrt{\alpha^2 4^j + (1-\alpha^2) 2^{2 \lfloor j/2 \rfloor}}.\]
It is monotone in both arguments, 
and furthermore satisfies the inequality
\begin{equation} \label{eq:rhogrowth}
\rho \left(\frac{\alpha}{2},j+2\right) 
\geq 
\sqrt{2}  \rho \left(\alpha,j\right) 
.
\end{equation}
If $\xi \in V_{s,t}^0$,
 we get the following bounds for the values $r'$ and $\theta'$:
\begin{align}
& 2^s \rho(2^{-(t+1)},j) \leq r' \leq 2^{s+1}\rho(2^{-t},j) \label{eq:circsec1} \\
& \theta' \leq \tan \theta' 
= 
2^{- \lfloor j/2 \rfloor } \tan \theta 
\leq 
2^{- \lfloor j/2 \rfloor } \frac{\theta}{\cos \theta} 
\leq 
2^{- \lfloor j/2 \rfloor } 2^{t} \pi. \label{eq:circsec2}
\end{align}
Therefore, $D_j$ maps the sector $V_{s,t}^0$ into a circular sector
of angle $2^{\lfloor -j/2 \rfloor } 2^{t} \pi$ and radius bounded by \eqref{eq:circsec1}. 
In particular, at most $2^{t}$ of the rotated
sectors $R_{j,\ell} D_j V_{s,t}^0$ contain a given  point $\xi$
(in other words, if $\ell> 2^t$, then $R_{j,\ell}$ maps the $x$-axis $\theta=0$ 
to the line $\theta = \pi {\ell} 2^{-\lfloor j/2\rfloor}> 2^{\lfloor -j/2 \rfloor } 2^{t} \pi $, and, hence, 
$R_{j,\ell} D_j V_{s,t}^0 \cap  D_j V_{s,t}^0 = \emptyset$). This allows us to obtain the bound
\[
\sum_{j=0}^{\infty}  \sum_{\ell = 0}^{2^{\lfloor j/2\rfloor}-1} 
  |\chi_{V_{s,t}^0}(D_{-j} R_{j,\ell}^*\xi)| 
= 
\sum_{j=0}^{\infty} \sum_{\ell = 0}^{2^{\lfloor j/2\rfloor}-1} 
  |\chi_{ R_{j,\ell} D_j V_{s,t}^0}(\xi)| 
\leq 
  2^t \sum_{j=0}^{\infty} 
   |\chi_{ C_{j,s,t}}(\xi)|,
\]
where $C_{j,s,t}$ denotes the annulus
\[ C_{j,s,t} = \{\xi \in \Rr^2 \ | \ 2^s \rho(2^{-(t+1)},j) \leq |\xi| \leq 2^{s+1}\rho(2^{-t},j)\}.\]
If
$C_{j,s,t} \cap C_{j',s,t} \neq \emptyset$ and $j < j'$, then the inequality \eqref{eq:rhogrowth} give
\[
2^s \rho(2^{-(t+1)},j') \leq 2^{s+1} \rho(2^{-t},j) \leq 2^{s} 
\rho(2^{-t-2},j+4) \leq 2^{s} 
\rho(2^{-(t+1)},j+4).
\]
(The fact that the inner radius of $C_{j',s,t}$ is less than the outer radius of $C_{j,s,t}$ gives the first inequality.
The second inequality follows form  two applications of (\ref{eq:rhogrowth}).
The third inequality follows from the fact that $\rho(\cdot, j+4)$ is increasing.)
This estimate, together with the fact that $\rho(2^{-(t+1)},\cdot)$ is increasing, implies that $j' \leq j + 4$. 
Consequently, a point $\xi$ can be an element of at most $9$ different sets 
$C_{j,s,t}$, 
and
\[ \| \chi_{V_{s,t}^0}\|_* 
\leq 
\sup_{\xi\in \Rr^2}
2^t \sum_{j=0}^{\infty} 
   |\chi_{ C_{j,s,t}}(\xi)|
\leq
9 \times2^t.\]
\end{proof}

\appendix

%%%%%%%%%%%%%%%%%%%%%%%%%%%%%%
%%%%%%%%%%%%%%%%%%%%%%%%%%%%%%
%
%Appendix: Curvelet Construction
%
%%%%%%%%%%%%%%%%%%%%%%%%%%%%%%
%%%%%%%%%%%%%%%%%%%%%%%%%%%%%%
\section{Appendix: Curvelets}\label{append} 
Let $w_0$ and $w$ be two univariate functions that satisfy
$|w_0(r)|^2+\sum_{j=1}^{\infty} |w(2^{-j} r)|^2 = 1$ 
where $w$ is supported on $[\frac{2\pi}{3},\frac{8\pi}{3}]$ 
and $w_0$ is supported on $[0,\frac{8
\pi}{3}]$.
In addition, let  $\nu\in C_c^{\infty}(\Rr)$ be
supported in $[-\pi,\pi]$, and  satisfy
$\sum_{\ell\in\Zz} |\nu(\theta - \pi \ell)|^2 = 1$. 
With $\xi=(r\cos\theta,r\sin\theta)\in \Rr^2$,
we define a sequence $(\chi_j)_{j\in\Nn}$ of bivariate functions
by
$\chi_0(\xi) := w_0(r),$
$\chi_1(\xi) := w_1(r/2),$
and
$$\chi_{j} (\xi)
:= 
w(2^{-j}r)
\bigl[\nu(2^{\lfloor j/2\rfloor} \theta) +\nu\bigl(2^{\lfloor j/2\rfloor}(\theta-\pi)\bigr)\bigr],\text{ for }j\ge 2.
$$
For $j\ge 2$, each $\chi_j$ is supported in the wedge
$$
\left\{\xi\in \Rr^2 
\mid 
\frac{2\pi}{3}2^{j} \le r \le \frac{8\pi}{3}2^{j}, \ |\theta |\le 2^{-\lfloor j/2\rfloor} {\pi}\text{ or }
|\theta - \pi |\le 2^{-\lfloor j/2\rfloor}{\pi} \right\},$$
and we have
$\sum_{j=0}^{\infty} \sum_{\ell =0}^{2^{\lfloor j/2\rfloor}-1} |\chi_{j}(R_{j,\ell}^*\xi)|^2 = 1$.
In particular, for Fourier transform
$f\mapsto \widehat{f}$ defined (for integrable functions) 
as $\widehat{f}(\xi)= \frac{1}{(2\pi)^{d/2}} \int_{\Rr^d}f(x)e^{-ix\cdot \xi} \dif x$,
we have
\begin{equation}\label{square_pou}
\|f\|_{L_2}^2 = \sum_{j=0}^{\infty} \sum_{\ell =0}^{2^{\lfloor j/2\rfloor}-1} 
\int_{\Rr^2}
|\widehat{f} (\xi)\chi_{j}(R_{j,\ell}^*\xi)|^2\dif \xi. 
\end{equation}

\begin{figure}[ht]
\centering
 \includegraphics[height=2.5in]%{new_wedge.eps}
  {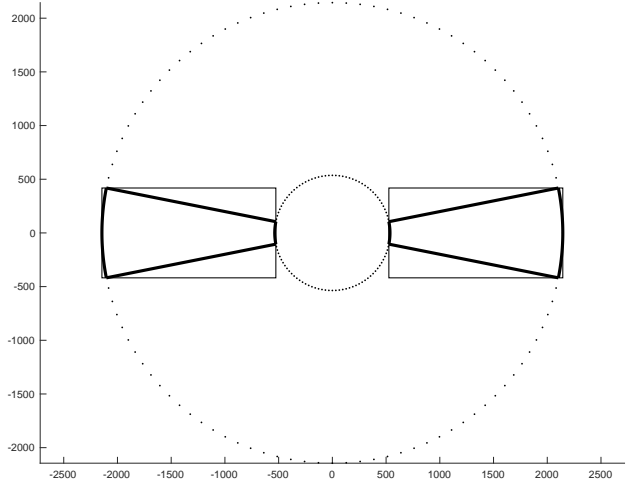}
  \caption{The support of $\widehat{\gamma_{j,0,0}}$ for $j=8$ (with thick boundary), contained in $I_j$.}\label{curvelet_pic}
  \end{figure}

An exercise in planar geometry shows that for $j\ge 2$, ${\rm supp}\,\chi_j$ 
is contained in the union of disjoint rectangles 
$$
I_j 
:= 
     \left(
        \left[ -\frac{8\pi}{3}2^j,-t_j\right]
        \times 
        \left[-\frac{\lambda_j}{2},\frac{\lambda_j}{2}
        \right]
     \right)
\cup
    \left(
      \left[t_j,\frac{8\pi}{3}2^j\right]
      \times 
      \left[-\frac{\lambda_j}{2},\frac{\lambda_j}{2}\right]
    \right)
$$
where 
$\frac{\lambda_j}{2} := 2^j \frac{8\pi}{3}\sin(2^{-\lfloor j/2\rfloor}{\pi})$ and
$t_j := 2^{j}\frac{2\pi}{3}\cos(2^{-\lfloor j/2\rfloor} {\pi})$.
Furthermore, we set 
$\Lambda_j:= 2(\frac{8\pi}{3}2^j-t_j)$, so that
$t_j +\frac{\Lambda_j}{2} := \frac{8\pi}{3}2^j$ (this is illustrated in Figure \ref{curvelet_pic} below). 
It follows that $\lambda_j$ is the total height of the support of $\chi_j$  and $\Lambda_j$ is the total width.
Also, $\chi_0$ is supported on the square $[-\frac{8\pi}3,\frac{8\pi}{3}]^2$,
 and $\chi_1$ is supported on $[-\frac{16\pi}{3},\frac{16\pi}{3}]^2$.
 For notational convenience, 
 we define 
$\Lambda_0:=\lambda_0:= \frac{16\pi}{3}$
and 
$\Lambda_1:=\lambda_1:= \frac{32\pi}{3}$.

The $j$-dependent grids are then defined as
$\grid_{j} :=  \{\left(\frac{k_1}{\Lambda_j},\frac{k_2}{\lambda_j}\right)\mid k\in 2\pi \Zz^2\}$. 
%%%%%%%%%%%%%%%%%%%%%%
%
%
%%%%%%%%%%%%%%%%%%%%%%
\begin{remark}\label{grid_correction_remark}
For $j\in \Nn$, define  $\delt_1(j) := \Lambda_j 2^{-j} $ and 
$\delt_2(j) := \lambda_j {2^{-\lfloor j/2\rfloor}}$. Define the grid correction
matrix $D_{\delt}^{(j)}$ as 
$D_{\delt}^{(j)}: = \begin{pmatrix} \delt_1(j)&0\\0&\delt_2(j)\end{pmatrix}$.
Then 
$
\begin{pmatrix} \Lambda_j &0\\0&\lambda_j\end{pmatrix} 
= 
D_j 
D_{\delt}^{(j)}
$, and we have 
$$\grid_j =  2\pi (D_{\delt}^{(j)})^{-1} D_j^{-1} \Zz^2.$$

Note that, for $j\ge 2$, $\delt_1(j) = \frac{4\pi}{3}[ 4 - \cos(2^{-\lfloor j/2\rfloor}{\pi})]$, 
and  $\delt_2(j) = {2^{-\lfloor j/2\rfloor}}[\frac{16\pi}{3} 2^j \sin (2^{-\lfloor j/2\rfloor} \pi)]$.
This implies, indeed, that $\delt_\ell$, $\ell=1,2$, are uniformly bounded
from above and away from $0$: The estimates
 \begin{equation}\label{correction_bounds}
 \begin{array}{lclc} 
 4\pi &\le \delt_1(j) &\le \frac{16\pi}{3}, \\
\frac{32\pi}{3}&\le \delt_2(j) &\le \frac{32\pi^2}{3}.% {3}/({32\pi}).
\end{array}
\end{equation}
hold.
\end{remark}

\subsubsection*{Curvelets} 
The curvelet $\gamma$ with $\mu(\gamma) =(j,\ell, k)$, with $j\in \Nn$, $\ell \in \Nn \cap [0,2^{\lfloor j/2\rfloor}-1]$
and $k\in \grid_j$, is defined as
$$
\gamma:=\gamma_{j,\ell, {k}}(x) 
:= 
\frac{1}{\sqrt{\Lambda_j\lambda_j}} (\mathcal{F}^{-1} \chi_j)(R_{j,\ell}^*x -  {k}).
$$
Note that if $j=0$, then $\ell = 0$. 
If we set $\gam^{(j)} := |D_j|^{-1/2}\gamma_{j,0, {0}} (D_{-j} \cdot)$,
then equation (\ref{eq:unitary_gam}), namely
$\gamma_{j,\ell, {k}} = |D_j|^{1/2} \gam^{(j)}(D_j(R_{j,\ell}^* \cdot - k))$,
follows automatically.
Indeed, we have 
\begin{equation}\label{generator_fourier_transform}
\widehat{\gam^{(j)}}(\xi) =
|D_j|^{1/2} \widehat{\gamma_{j,0,0}}(D_j \xi) =  \sqrt{\frac{|D_j|}{\Lambda_j \lambda_j}}\chi_j ( D_j\xi).
\end{equation}

At this point, we define
$$
\Dcurv := 
 \{\gamma_{0,0,k}\mid k\in \grid_0\}
 \cup 
  \{\gamma_{1,0,k}\mid k\in \grid_1\}
 \cup
\left\{
  \gamma_{j,\ell,k}\mid j\ge 2, \ell\in \{0,\dots ,2^{\lfloor j/2\rfloor}-1\}, k\in \grid_j
\right\}
.
$$

From this we have the following two results.
%%%%%%%%%%
%
%%%%%%%%%%
\begin{lemma}\label{support}
Let $A := \frac{\sqrt{2}\pi}{3}$ and  $B := \frac{16\pi^2}{3}$.
For all $j$,
$\mathrm{supp}(\widehat{\gam^{(j)}})  \subset  [-B,B]^2$. For $j\ge 4$,
$\widehat{\gam^{(j)}}(\xi)= 0 $ when $|\xi|\le A$.
\end{lemma}
%%%%%%%%%%
%
%%%%%%%%%%

%%%%%%%%%%
%
%%%%%%%%%%
\begin{proof}
This follows from (\ref{generator_fourier_transform}); namely, the fact that
$\mathrm{supp}(\widehat{\gam^{(j)}})  = D_{-j}\, \mathrm{supp}(\chi_j )$,
which, together with some of the construction details before,
gives $B = \max_{j} 2^{-\lfloor j/2\rfloor} \frac{\lambda_j}{2}$ and
$A = \min_{j\ge 4} 2^{-j} t_j$.
\end{proof}
%%%%%%%%%%
%
%%%%%%%%%%
%
%
%
%%%%%%%%%%
%
%%%%%%%%%%
\begin{lemma}\label{schwartz_bound}
For each multi-index $\alpha =(\alpha_1,\alpha_2)$, there is $C_{\alpha}$ 
so that for all $j\ge 4$
$$\|D^{\alpha}\widehat{\gam^{(j)}}\|_{\infty} \le C_{\alpha}.$$
\end{lemma}
%%%%%%%%%%
%
%%%%%%%%%%
%
%
%
%%%%%%%%%%
%
%%%%%%%%%%
\begin{proof}
By (\ref{generator_fourier_transform}),
it suffices to consider $\chi_j(D_j\cdot)$ in place of $\gam^{(j)}$, since
$ \sqrt{\frac{|D_j|}{\Lambda_j \lambda_j}} =\sqrt{\delt_1(j) \delt_2(j)}$ 
is bounded above and below by Remark \ref{grid_correction_remark}.
We have that
$$
\chi_j (D_j\xi) 
= 
 w\left(\sqrt{\xi_1^2 + (2^{-\lceil j/2 \rceil} \xi_2)^2}) \right)
 v\bigl(2^{\lfloor j/2\rfloor} \tan^{-1} \bigl(2^{-\lceil j/2\rceil } {\xi_2}/{\xi_1}\bigr)\bigr).
$$
The result follows from the smoothness of the functions involved, 
and the fact that $|\xi_1|,|\xi_2|\le B$ and $|\xi_1|>A>0$.
\end{proof}
%%%%%%%%%%
%
%%%%%%%%%%

\bibliographystyle{plain}
\bibliography{EHR}
\end{document}